\documentclass[leqno]{amsart}
\usepackage{stackengine}
\usepackage{amsthm}
\usepackage{amssymb}
\usepackage{mathrsfs}
\usepackage{graphicx}
\usepackage{array}
\usepackage{todonotes}
\usepackage{xcolor}
{%
\setcounter{enumi}{0}

\begin{enumerate}}%
{\end{enumerate} }

{%
\setcounter{enumi}{0}

\begin{enumerate}}%
{\end{enumerate} }

 \newcommand{\grad}{\nabla}

\newcommand{\rn}{\mathbb{R}^{N}}
\newcommand{\R}{\mathbb{R}}
\newcommand{\zn}{\mathbb{Z}^{N}}
	\newcommand{\Z}{\mathbb{Z}}

\newcommand{\A}{\alpha}

\newcommand{\B}{\beta}
\newcommand{\del}{\partial}

\newcommand{\LL}{\mathcal{L}}

\newcommand{\bj}{\mathbf{j}}
\newcommand{\bi}{\mathbf{i}}

\def\N{\mathbb{N}}

\usepackage[colorlinks]{hyperref}

\newtheorem{thm}{Theorem}[section]
\newtheorem{lem}[thm]{Lemma}
\newtheorem{prop}[thm]{Proposition}

\theoremstyle{definition}

\newtheorem{rem}[thm]{Remark}

\numberwithin{equation}{section}

\newcommand{\cm}{\color{black}}
\newcommand{\cb}{\color{black}}

\newcommand{\nc}{\normalcolor}

\DeclareRobustCommand{\SkipTocEntry}[5]{}

\begin{document}

\title[Powers of discrete Laplacian for HJB equations]
{Discretization of fractional fully nonlinear equations by powers of discrete Laplacians}

\author{Indranil Chowdhury}
\address{Indian Institute of Technology - Kanpur, India}
\curraddr{}
\email{indranil@iitk.ac.in}

\author{Espen R. Jakobsen}
\address{Norwegian University of Science and Technology, Norway}
\curraddr{}
\email{espen.jakobsen@ntnu.no}

\author{Robin Ø. Lien}
\address{Norwegian University of Science and Technology, Norway}
\curraddr{}
\email{robin.o.lien@ntnu.no}





\date{}

\dedicatory{}

\begin{abstract}
We study discretizations of fractional fully nonlinear equations by powers of discrete Laplacians. Our problems are parabolic and of order $\sigma\in(0,2)$ since they involve fractional Laplace operators $(-\Delta)^{\sigma/2}$. They arise e.g.~in control and game theory as dynamic programming equations \cb -- HJB and Isaacs equation -- \nc and solutions are non-smooth in general and should be interpreted as viscosity solutions. 
Our approximations are realized as finite-difference quadrature approximations and are 2nd order accurate for all values of $\sigma$. The accuracy of \cb previous approximations of fractional fully nonlinear equations \nc
depend on $\sigma$ and are worse when $\sigma$ is close to $2$.
We show that the schemes are monotone, consistent, $L^\infty$-stable, and convergent using a priori estimates, viscosity solutions theory, and the method of half-relaxed limits. We \cb also prove a second order error bound for smooth solutions and \nc present \cb many \nc numerical examples.
\end{abstract}

\maketitle


\section{Introduction}
\label{sec::intr}
In this paper we introduce and analyze numerical schemes based on powers of the discrete Laplacian in the context of nonlocal fully nonlinear equations. 
Our equations involve fractional Laplacians, pseudo-differential operators that can be defined equivalently as Fourier multipliers, singular integral operators, or  powers of the Laplacian \cite{Ap:Book,Kw}: For $\sigma\in(0,2)$,
\begin{align}
-(-\Delta)^{\frac\sigma2}\phi(x)&=\mathcal{F}^{-1}\big[-|\xi|^{\sigma}\hat\phi(\xi)\big](x)\nonumber\\
&=\lim_{r\to 0}\int_{|z|>r} \Big( \phi(x+ z) - \phi(x) \Big) \, \frac{c}{|z|^{N+\sigma}}dz\nonumber\\ 
&=
 \frac{1}{|\Gamma(-\frac{\sigma}{2})|} \, \int_0^{\infty} \Big( e^{t\Delta} \phi(x) - \phi(x) \Big) \, \frac{dt}{t^{1+ \frac{\sigma}{2}}},
 \label{fraclap}
\end{align} 
\cb where $c=\frac{2^{\sigma}\Gamma(\frac{N+\sigma}{2})}{\pi^{N/2}|\Gamma(-\frac{\sigma}{2})|}.$ \nc We discretize this operator by powers of the discrete Laplacian, \begin{align}\label{dlap}\Delta_h\phi(x)=\sum_{k=1}^N\frac{\phi(x+he_k)-2\phi(x)+\phi(x-he_k)}{h^2},
 \end{align}
denoted by $-(-\Delta_h)^{\frac\sigma2}$ and obtained from  \eqref{fraclap} by replacing $\Delta$ by $\Delta_h$  \cite{ciaurri2015fractional}, cf. \eqref{discrete_fraclap}.

The equations we consider are fully-nonlinear, possibly (strongly) degenerate equations from optimal control and differential game theory, equations with a large number of applications in engineering, science,  economics, etc. \cite{Be:book,FS:book,OS:book,Ha:book}:
\begin{align}\label{eq:Isaacs}
    u_t -F(x,t, Du,-(-\Delta)^{\frac\sigma2}u)= 0\mbox{\quad in} \quad \rn\times(0,T]=:Q_T,
\end{align}
or more generally, Hamilton-Jacobi-Bellmann(HJB)/Isaacs equations
\begin{align}\label{eqn:main}
 &  u_t- \inf_{\B\in\mathcal{B}}\sup_{\A\in\mathcal{A}} \left\{ \LL^{\A,\B}u-c^{\A,\B}(x,t) u+ f^{\A,\B} (x,t)\right\}  = 0  \mbox{\quad in} \quad Q_T, 
 \end{align}
where $\mathcal{A}$, $\mathcal B$ are compact sets and $\LL^{\A,\B}$ is the $\sigma\wedge 1$ order  drift-diffusion operator
\begin{align*} 
\LL^{\A,\B}\phi(x) := -a^{\A,\B}(x,t)(-\Delta)^{\frac\sigma2}\phi(x) + b^{\A,\B}(x,t)\cdot D\phi(x).
\end{align*}
Equation \eqref{eqn:main} is the dynamic programming equation for a finite horizon optimal stochastic differential game \cite{Bi,OS:book,Ha:book}, see Section \ref{sec:DG} for the details.

Equation  \eqref{eq:Isaacs} can be degenerate parabolic as we allow $F$ to be non-decreasing in last variable. The solutions are not smooth in general. Even for nonconvex uniformly parabolic problems, the solutions could be too irregular for the equation to hold pointwise.
The correct notion of (weak) solution for this type of problems is viscosity solutions \cite{JK05,JK06,BI08}. Wellposedness, regularity,  and other properties of viscosity solutions for nonlocal fully nonlinear PDEs has been intensely studied in recent years. Regularity in the degenerate case comes from comparison type of arguments and typically gives preservation of the regularity of the data \cite{JK06}. Solutions are therefore often no more than Lipschitz continuous, see Section \ref{sec:ex1} for an example.

There is an extensive literature on numerical methods for local fully nonlinear equations including finite differences, semi-Lagrangian, finite elements, spectral, Monte Carlo, 
and many more, see e.g. \cite{Crandall-lions,Souga, DK:book,FF:book,BS91, Camilli-Falcone, Le00, BH:2001,BZ03,Deb_Jak:13,SS:2013,BT2004,EHJ2017}.
Here there is the added difficulty of
discretizing the fractional and nonlocal operators in a monotone, stable, and
consistent way. These operators are singular integral operators, and
can be discretized by quadrature after truncating the singular part and correcting with a
suitable second derivative term \cite{Cont-2005}. In the setting of nonlocal Bellman-Isaacs equations, such approximations were introduced in
\cite{JKL08,C-Jakobsen2009,BJK1}
with further developments, including error estimates, in e.g.
\cite{BCJ1, Coc-Risebro16, Reisinger2021, Roxana-Reisinger2021, BJK2,BJK1,JKL08}.  
These approximations have fractional order accuracy, depending on the order of the fractional/nonlocal operator.


The numerical approximations 
used in this paper are based on powers of discrete Laplacians. As opposed to the approximations above, they are 2nd order accurate regardless of the order of the operators. They can also be interpreted as quadrature rules and represented as an infinite series expansion with explicit weights \cite{Ciaurri-Stinga-02}. These weights satisfy a discrete version of the L\'evy integrability condition. 
Previously powers of discrete Laplacians have been used to discretize linear equations \cite{CDG20}, porous medium equations \cite{EJT18b}, and very recently also certain HJB equations \cite{CJ23}. \cb We also refer to \cite{DJ24} for results on 1d quasilinear equations related to integrated porous medium equations with fractional pressure. \nc In \cite{CJ23}, (optimal fractional) error bounds for numerical schemes for convex fractional equations are studied. But in the case of powers of the discrete Laplacian, only very simple non-degenerate constant coefficient problems were considered, and no numerical experiments were performed. This paper gives extensions of the schemes and convergence results in \cite{CJ23} to a very large class of fully nonlinear equations, including non-convex, strongly degenerate, and variable coefficients problems. We show that the resulting schemes are consistent, monotone, stable, and convergent. We do not study error bounds, but we perform a number of numerical experiments.

To simplify the presentation we introduce the numerical scheme and perform the detailed convergence analysis for the following version of the problem:  
\begin{equation}
\begin{aligned}
    u_t -F\big(-(-\Delta)^{\frac{\sigma}{2}}u\big)&=f(x,t),  \hspace{5mm}&(x,t)&\in Q_T, \\
    u(x,0)&= u_0(x), &x&\in\mathbb{R}^N.
\end{aligned}    
\label{eq::simple_parabolic_eq}
\end{equation}
We also focus on explicit schemes using forward Euler time discretizations. Under suitable CFL conditions, we then show that the schemes are monotone satisfying a comparison principle and $L^\infty$-stable. We use viscosity solutions and the Barles-Perthame-Souganidis method of half-relaxed limits \cite{BS91} to show that solutions of the schemes converge uniformly to solutions of the equation. To do that, we show that the particular versions of monotonicity and consistency required by \cite{BS91} are indeed satified by our schemes. Numerical examples are presented for problem \eqref{eq::simple_parabolic_eq}, in one and two dimension, for problems with non-smooth and smooth solutions. We also illustrate numerically the convergence of solutions as $\sigma\to0^+$ and $\sigma\to2^-$, showing that our schemes are stable also with respect to these limits. 
Later (Section \ref{sec::extn}) we explain how to extend the schemes and results to other time-discretisations and more general problems, including problems with first order/convection terms, and the Bellman-Isaacs equation \eqref{eq:Isaacs}.  


\medskip
 The remaining part of this paper is organized as follows: In Section \ref{sec:assump_wellposed} we introduce the notation and assumptions, and give  well-posedness results for equation \eqref{eq::simple_parabolic_eq}. In addition, we discuss the relation between HJB-Isaacs equations and a zero sum game. In Section \ref{sec:fraclap} we give the \cb analytical \nc results for approximations based on powers of discrete Laplacians \cb -- showing monotonicity, stability, consistency, and  convergence of the schemes. We also prove an error bound showing that the scheme is 2nd order in $h$ when solutions are smooth enough. \nc Numerical examples are presented in Section \ref{sec::num_exp}, and Section \ref{sec::extn} covers extensions of the results for various cases -- other time discretizations, equations involving convection \cb and more general diffusion terms, and \nc HJB-Isaacs type equations. 


\section{On 
nonlocal PDEs} 
\label{sec:assump_wellposed} 

In this section we present the  assumptions on the nonlocal fully nonlinear equation \eqref{eq::simple_parabolic_eq} and give wellposedness and regularity results. In the second part we explain the connection between the HJB/Isaacs equation \eqref{eqn:main} and a stochastic differential game.

Let us first introduce some notation. By $C,K$ etc. we mean various constants which may change from line to line. \cb We let $(\cdot)^+ := \max\{0,\cdot\}$, $|\cdot|$ denote the euclidean norm, and also define the norms $\|u\|_{L^\infty}= \sup _{x} |u(x)|$ and $\|u\|_{W^{1,\infty}} = \|u\|_{L^\infty} + \sup_{x\neq y} \frac{|u(x)-u(y)|}{|x-y|}$. \nc 
\cm Moreover, $C^{n}(Q)$ ($C_b^n(Q)$) for $n \in \N$ denotes the space of $n$ times continuously differentiable functions on $Q$ (with bounded derivatives). \nc 

\subsection{Wellposedness of nonlocal PDEs} We will study viscosity solutions of equation \eqref{eq::simple_parabolic_eq} under the following assumptions:
 
\makeatletter
\newcommand{\myitem}[1]{%
\item[#1]\protected@edef\@currentlabel{#1}%
}
\makeatother
\begin{enumerate}
\medskip\myitem{$\mathbf{(A1)}$}\label{F1} $F(l_1)-F(l_2) \leq L_F(l_1-l_2)^+$ for all $l_1, l_2 \in \R$.
  \bigskip  
\myitem{$\mathbf{(A2)}$}\label{F3} $f\in C_b(Q_T)$ and $u_0\in C_b(\R^N)$.
\end{enumerate}
 \medskip

Assumption \ref{F1} implies that $F$ is both Lipschitz continous (with Lipschitz constant $L_F$) and nondecreasing.
A definition and general theory of viscosity solution for the nonlocal equations like \eqref{eqn:main} can be found e.g. in \cite{JK05,BI08}, but we do not need this generality here. In particular since there is no local diffusion, we could follow the simpler (comparison) arguments of \cite{CJ17}. 

We have the following strong comparison and well-posedness  results for \eqref{eq::simple_parabolic_eq}:
\begin{prop} \label{result:welposed-visco}
    Assume \ref{F1} and \ref{F3}.
    \medskip
    
    \noindent (i) (Comparison) If $u\in \text{USC}_b$ and $v\in\text{LSC}_b$ are bounded viscosity subsolution and supersolution of \eqref{eq::simple_parabolic_eq} respectively,
    then 
    $$u \leq v \text{ in } \mathbb{R}^N \times [0,T].$$
        
    \noindent (ii) (Existence and uniqueness) There exists a unique bounded continuous viscosity solution $u$ of \eqref{eq::simple_parabolic_eq}.
    \medskip
    
    \noindent (iii) ($L^\infty$-stability) The solution $u$ in (ii) satisfies: $$\|u(\cdot, t)\|_{L^{\infty}} \leq \|u_0\|_{L^{\infty}}+t\big(\|f\|_{L^{\infty}} + |F(0)|\big).$$

    \label{prop::comparison_existence_stability}
\end{prop}
\begin{proof}
    We refer to \cite[Section 6]{CJ17} for 
    part (i) and (ii) in the case when $u_0$ and $f$ are $BUC$. As mentioned there, the parabolic proof is a simple generalization of the detailed proofs in the elliptic case in \cite[Theorem 2.1, 2.3 and Corollary 2.2]{CJ17}. These proofs easily extend to the case when $u_0$ and $f$ are $C_b$.\footnote{To do this we need to modify the viscosity solution doubling of variables argument in the following way: First pass the limit to undo the doubling and then to undo the penalisation of infinity. The proof still works, because as long as we penalise infinity, we are working on a compact set, and $C_b$ implies $BUC$ here. $UC$ is only needed to undo the doubling (the first limit). This order works for HJB/Isaacs type of equations, but in \cite{CJ17} the order of the limits needs to be opposite because quasi-linear operators are considered.}
    \newline\hspace*{1em}For (iii), begin by defining
    \begin{equation*}
w(x,t)=\|u_0\|_{L^{\infty}(\rn)}+t(|F(0)|+\|f\|_{L^{\infty}(Q_T)}).
    \end{equation*}
    Inserting $w$ into  the left-hand side of \eqref{eq::simple_parabolic_eq}, we find that it is a supersolution of the equation:
    \begin{align*}
        w_t-F(-(-\Delta)^{\frac{\sigma}{2}}w) &= (|F(0)|+\|f\|_{L^{\infty}})-F(0)  \geq f,
    \end{align*}
    where we get $F(0)$ in the first equality since $w$ is independent of $x$. Similarly, $-w$ is a subsolution of \eqref{eq::simple_parabolic_eq}. Then the result follows by part (i). 
\end{proof}

 



\subsection{A differential game related to nonlocal PDEs}\label{sec:DG}
The HJB/Isaacs equation \eqref{eqn:main} is related to a zero sum differential game where players control the following SDE \cite{Ap:Book,CT:Book} driven by a $\sigma$-stable Levy process of the form  
\begin{align}\label{SDE}dX_s = b^{\alpha_s,\B_s}(X_s,s)\,ds+\int_{|z|>0} \eta^{\alpha_s,\B_s}(X_{s^-},s)z \,\tilde N(dz,ds), \quad X_t=x,
\end{align}
when $\eta^{\A,\B}(x,t)=[a^{\A,\B}(x,t)]^{\frac1{\sigma}}$ and $\tilde N$ is the compensated Poisson random measure. 

The  Poisson random measure $N(B,t)$ counts the number of jumps $z\in B$ of the driving process up to time $t$ \cite{Ap:Book,OS:book}. For $\sigma$-stable processes, $\mathbb E\big[ N(dz,dt)\big]=\nu(dz)dt$ with $\nu(dz)=\frac{c\, dz}{|z|^{N+\sigma}}$, \cb where $c$ is as in \eqref{fraclap}, \nc and then the compensated measure $\tilde N(dx,dt)=N(dz,dt)-\nu(dz)dt$.
Note that by self-similarity of the definition of the fractional Laplacian and the definition of $\eta^{\A,\B}$,
$$-a^{\A,\B}(x,t)(-\Delta)^{\frac\sigma2}\phi(x)=p.v.\int_{|z|>0} \big( \phi(x+\eta^{\A,\B}(x,t)z) - \phi(x) \big) \, \frac c{|z|^{N+\sigma}}dz,$$
and hence the generator of $X_s$ in \eqref{SDE} is the operator $\LL^{\A,\B}$ in \eqref{eqn:main} \cite{Ap:Book,CT:Book}.

The game setting is a zero-sum game with two players, 
separate controls $\alpha_\cdot$ and $\B_\cdot$ belonging to  sets of admissible controls $\mathcal A_{\text{ad}}$ and $\mathcal B_{\text{ad}}$, and a ``cost" function
\begin{align*}
J(x,t,\alpha_\cdot,\B_\cdot)&=E\Big[\int_t^T e^{-\int_t^s c^{\alpha_r,\beta_r}(X_r,r)dr}f^{\alpha_s,\B_s}(X_s,s)\,ds \\ 
& \hspace{4cm}+e^{-\int_t^T c^{\alpha_r,\beta_r}(X_r,r)dr}u_0(X_T)\Big], 
\end{align*}
where $c$, $f$, and $u_0$ are the discounting rate, running cost, and terminal cost respectively. The ``cost" function is a cost for one player who seeks to minimise it and a gain for the other who seeks to maximise it. The game can be understood from the (upper/lower) values of the game  defined as
$$u(x,T-t)=\inf_{\alpha_\cdot\in \mathcal A_{\text{ad}}}\sup_{\B_\cdot\in \mathcal B_{\text{ad}}} J(x,t,\alpha_\cdot,\B_\cdot).$$
In the dynamic programming approach to optimal control and differential games \cite{Bi,OS:book,Ha:book}, this function is shown to satisfy the HJB/Isaacs equation \eqref{eqn:main} with initial data $u_0$.

\section{Discretization by powers of discrete Laplacian} \label{sec:fraclap}

In this section we approximate the nonlocal fully nonlinear HJB/Isaacs equation \eqref{eq::simple_parabolic_eq} using a forward Euler approximation in time and powers of the discrete Laplacian to approximate the fractional Laplacians. We then show that the resulting scheme is  consistent, monotone, and $L^\infty$-stable. Using the method of half-relaxed limits of Barles-Perthame-Souganidis \cite{BS91}, we then show convergence of the method toward the viscosity solution of \eqref{eq::simple_parabolic_eq}.

 We introduce space and time grids, $\mathcal{G}_h = h \mathbb{Z}^N=\{x_{\bj} = h\bj : \bj \in \mathbb{Z}^N\}$ and $\mathcal{T}_\tau^T = \{t_n := n\tau\}_{n=0}^M$, for $\tau = \frac{T}{M}>0$, $M\in \mathbb{N}$. The parameters $h$ and $\tau$ are then the distance between the grid points in the two grids.

\subsection{Powers of the discrete Laplacian}

Let $\Delta_h\phi(x)$ be the discrete Laplacian, the 2nd order central finite difference approximation of $\Delta\phi$ defined in \eqref{dlap}. Then the powers of the discrete Laplacian \cite{Ciaurri-Stinga-02,  EJT18b} is defined as 
\begin{align}\label{discrete_fraclap}
-(- \Delta_h)^{\frac{\sigma}{2}} \phi(x) := \frac{1}{|\Gamma(-\frac{\sigma}{2})|} \, \int_0^{\infty} \Big( e^{t\Delta_h} \phi(x) - \phi(x) \Big) \, \frac{dt}{t^{1+ \frac{\sigma}{2}}}, \quad \sigma\in(0,2),
\end{align}  
where $U(t)=e^{t\Delta_h} \psi $ is the solution of semi-discrete heat equation 
\begin{equation}
    \begin{aligned}
\partial_t U(x,t) & = \Delta_h \,  U(x, t) \quad \mbox{for} \quad (x,t) \in \rn \times (0,\infty), \\
 U(x,0) & = \psi(x) \quad \mbox{for} \quad x \in \rn. 
    \end{aligned}  
    \label{eq:semi_disc_heat_eq}
\end{equation}
An explicit formula for $e^{t\Delta_h} \phi$ and details related to this approximation can be found in Section 4.5 of \cite{EJT18b}. 

The operator $-(- \Delta_h)^{\frac{\sigma}{2}}$ is a monotone (positive coefficients) operator given by a series expansion with explicit weights, and these weights satisfy a discrete version of the Levy integrability condition.

\begin{lem}\label{lem::discr_fraclap_quadrature}
Let $-(- \Delta_h)^{\frac{\sigma}{2}}$ be defined by \eqref{discrete_fraclap}. Then
\begin{align*}
-(- \Delta_h)^{\frac{\sigma}{2}} \phi(x) = \sum_{\bj \in \zn \setminus \{0\}} \Big( \phi(x+ x_{\bj}) - \phi(x)\Big) \kappa_{\sigma, h,\bj},
\end{align*}
where 
$$\kappa_{\sigma, h, \bj} = \frac{1}{h^{\sigma}}\frac{1}{|\Gamma(-\frac{\sigma}{2})|}\int_0^{\infty}G(\bj,t)\frac{\text{d}t}{t^{1+\frac{\sigma}{2}}},$$
$G(\bj,t) = e^{-2Nt}\prod_{i=1}^N I_{|\bj_i|}(2t)$, and $I_m\geq0$ is the modified Bessel function of the first kind and order $m \in \mathbb{N}$. 
Moreover, $\kappa_{\sigma, h, \bj}\geq0 $ for all $\bj\in \Z^N$ and there is a  $C_{\sigma}>0$ such that
$$
 \sum_{\bj \in \mathbb{Z}^N\backslash \{0\}}\kappa_{\sigma, h, \bj} = \frac{C_{\sigma}}{h^{\sigma}}.
 $$
\end{lem}

Note that the last part of the lemma (the formula for the sum) seems not to have been proved before, even though something like this is needed in \cite{EJT18b}.

\begin{proof}
    In the one dimensional case, the results follow from \cite[Theorem 1.1]{Ciaurri-Stinga-02}. For the proof of the quadrature representation, see \cite[Lemma 4.20]{EJT18b}.  

Note that $G\geq0$ and then also $\kappa_{\sigma, h,\bj}\geq0$. For the final result we use from \cite[Section 8.2.]{ciaurri2015fractional} that $\sum_{\bj \in \mathbb{Z}^N}G(\bj,t)=1$ to see that
    \begin{equation}
    \begin{aligned}
            \sum_{\bj \in \mathbb{Z}^N \backslash \{0\}}\kappa_{\sigma, h, \bj} & =\frac{1}{h^{\sigma}}\frac{1}{|\Gamma(-\frac{\sigma}{2})|}\sum_{\bj \in \mathbb{Z}^N\backslash \{0\}} \int_0^{\infty}  G(\bj,t)\frac{\text{d}t}{t^{1+\frac{\sigma}{2}}} \\ &= \frac{1}{h^{\sigma}}\frac{1}{|\Gamma(-\frac{\sigma}{2})|}\int_0^{\infty}  (1-G(0,t))\frac{\text{d}t}{t^{1+\frac{\sigma}{2}}}, \\
    \end{aligned}
    \label{eq::sum_of_weights_calculation}
    \end{equation}
    where we interchange the sum and integral by Tonelli's theorem.   
    
We must show that the integral in the above expression converges. To this end, write
    \begin{align*}
        \int_0^{\infty}  (1-G(0,t))\frac{\text{d}t}{t^{1+\frac{\sigma}{2}}} = \underbrace{\int_0^{1}  (1-G(0,t))\frac{\text{d}t}{t^{1+\frac{\sigma}{2}}}}_{:=I_1} + \underbrace{\int_1^{\infty}  (1-G(0,t))\frac{\text{d}t}{t^{1+\frac{\sigma}{2}}}}_{:=I_2}.
    \end{align*}
    $I_2$ clearly converges, since  
    \begin{align*}
        I_2 \leq \int_1^{\infty}\frac{1}{t^{1+\frac{\sigma}{2}}}dt < \infty,
    \end{align*}
    for $\sigma \in (0,2)$. 
    
Showing that $I_1$ converges requires a bit more work. Using the definitions of $G$ and of modified Bessel functions of the first kind we can write
    \begin{align*}
        G(0,t) = e^{-2Nt}\Big(1+\sum_{m=1}^{\infty}\frac{t^{2m}}{(m!)^{2}} \Big)^N.
    \end{align*}
    The infinite sum in the above expression is a power series with infinite radius of convergence. Consequently, $G(0,t)$ is a smooth function on the whole real line. In particular, $G(0,\cdot)\in C^1([0,1])$. By the mean value theorem, we then have that
    \begin{align*}
        \frac{G(0,t)-1}{t} = G'(s)\qquad\text{for some}\qquad s\in(0,t),
    \end{align*}
    where we use that $G(0,0)=1$. Let $M = \max_{s \in [0,1]} |G'(s)|$ (which exists since $G(0,\cdot)\in C^1([0,1])$). Consequently,
    \begin{align*}
        I_1 \leq \int_{0}^{1}\frac{M}{t^{\frac{\sigma}{2}}} dt < \infty,
    \end{align*}
    for $\sigma \in (0,2)$.

Since $I_1$ and $I_2$ converges, we have shown that 
    \begin{align*}
        \sum_{\bj \in \mathbb{Z}^N \backslash \{0\}}\kappa_{\sigma, h, \bj} = C_{\sigma}\frac{1}{h^{\sigma}},
    \end{align*}
    for some constant $C_{\sigma}>0$ and $\sigma \in (0,2)$. The constant $C_{\sigma}$ is \textit{strictly} greater than zero since the integrand in \eqref{eq::sum_of_weights_calculation} is positive almost everywhere on the domain of integration. 
\end{proof}

By Lemma 4.22 in \cite{EJT18b}, $-(- \Delta_h)^{\frac{\sigma}{2}}$ is a second order approximation of $-(- \Delta)^{\frac{\sigma}{2}}$:

\begin{lem}[\cite{EJT18b}]\label{lem:trun_err_fraclap}
Assume $\sigma \in (0 ,2)$ \cm and 
$\phi\in C_b^4(\R^N)$. Then 
\begin{align}
\label{fraclap_trunc}
\big| (- \Delta_h)^{\frac{\sigma}{2}} \phi(x) - (- \Delta)^{\frac{\sigma}{2}} \phi(x)\big| \leq \tfrac N{10}h^2\Big(
\|D^4 \phi \|_\infty+ \tfrac{2-\sigma}{2+\sigma}c_N
\|\phi \|_\infty\Big),
\end{align}
\cb where 
$c_N = \sum_{i=1}^N \int_{\R^N}\big| \partial_{z_i}^{(4)}K_N(z,1) \big| \ dz$ 
and 
$K_N(x,t)$ 
 is the heat kernel in $\R^N$.
\end{lem} 

\cb This is \cite[Lemma 4.22]{EJT18b} with precise values of the constants. \cm 
Note that the rate $O(h^2)$ is uniform over $\sigma\in(0,2)$.
\cb
We now redo the proof to get explicit values of the constants. 
\begin{proof}
 By the definition of the fractional Laplacian, we have that
    \begin{align*}
        \big|& (- \Delta_h)^{\frac{\sigma}{2}} \phi(x) - (- \Delta)^{\frac{\sigma}{2}} \phi(x)\big| \leq \frac{1}{|\Gamma(-\frac{\sigma}{2})|} \Big[\int_0^1 |E(x,t)| \frac{dt}{t^{1+\frac{\sigma}{2}}} +\int_1^\infty |E(x,t)| \frac{dt}{t^{1+\frac{\sigma}{2}}} \Big],
    \end{align*}
    for $E(x,t):= e^{t \Delta_h}\phi(x) -e^{t\Delta}\phi(x)$.
    Let $\tau(x,t):=\partial_t e^{t\Delta}\phi(x)-\Delta_h e^{t\Delta}\phi(x)$, and observe that since $e^{t\Delta_h}\phi(x)$ solves the semi-discrete heat equation \eqref{eq:semi_disc_heat_eq},
    \begin{align*}
        \partial_t E &= \Delta_h E -\tau.
    \end{align*}
    Since $\Delta_h$ is monotone and $E(x,0)=0$, comparison holds and leads to the bound 
    \begin{align*}
        \|E(\cdot,t) \|_\infty \leq \int_0^t \| \tau(\cdot, s) \|_\infty \; ds
    \end{align*}
($\pm$ the right hand side is a super/subsolution).
    By Taylor expansion,
    \begin{align*}
        | \tau(x,t) | \leq \frac{h^2}{12} \sum_{i=1}^N \max_x \big| \del^{(4)}_{x_i} e^{t\Delta}\phi(x) \big|, 
    \end{align*}
 and then by introducing the heat kernel $K_N$ we find that,
    \begin{align*}
         &\big|\del^{(4)}_{x_i} e^{t\Delta}\phi(x)\big| = \Big| \del^{(4)}_{x_i} \int_{\R^N} \phi(x-y)K_N(y,t) \; dy\Big| \\[0.1cm]
         &\leq \begin{cases}
             \|\partial_{x_i}^4 \phi \|_\infty \int_{\R^N} K_N(y,t) \; dy \\[0.2cm]
             \| \phi \|_\infty \int_{\R^N} |\del^{(4)}_{x_i}K_N(y,t)| \; dy
         \end{cases}
          \leq 
         \begin{cases}
              \|\partial_{x_i}^4 \phi \|_\infty \\[0.2cm]
              \| \phi \|_\infty \frac{c_N}{t^2},
         \end{cases}
    \end{align*}
    where $c_N = \int_{\R^N}\sum_{i=1}^N \big| \partial_{z_i}^{(4)}K_N(z,1) \big| \ dz$ and we have used $K_N\geq0$, $\int K_N dx=1$, and self-similarity  $K_N(x,t)=\frac{1}{t^{N/2}}K_N(\frac x{\sqrt{t}},1)$.
    For $0 \leq t \leq 1$, we use the first bound,
    \begin{align*}
        \|E(\cdot,t) \|_\infty \leq th^2\frac{N}{12}\|D^4 \phi\|_\infty. 
    \end{align*}
    When $t\geq 1$, we split the integral and use both bounds as follows:
    \begin{align*}
        \|E(\cdot,t) \|_\infty &\leq \int_0^1 h^2\frac{N}{12}\|D^4 \phi \|_\infty \; ds + \int_1^t h^2 \frac{N}{12} \| \phi \|_\infty \frac{c_N}{s^2} \; ds \\
        &=  h^2\frac{N}{12}\|D^4 \phi \|_\infty +h^2 \frac{N}{12} \| \phi \|_\infty c_N \Big(1-\frac{1}{t} \Big).
    \end{align*}
 We can now compute the finial estimate:
    \begin{align*}
        \big|& (- \Delta_h)^{\frac{\sigma}{2}} \phi(x) - (- \Delta)^{\frac{\sigma}{2}} \phi(x)\big| \\[0.1cm] 
        &\leq \frac{1}{|\Gamma(-\frac{\sigma}{2})|} h^2\frac{N}{12}\bigg[\int_0^1 t\|D^4 \phi \|_\infty \frac{dt}{t^{1+\frac{\sigma}{2}}}  +\int_1^\infty \Big( \|D^4 \phi \|_\infty +\| \phi \|_\infty c_N \big(1-\frac{1}{t} \big) \Big) \frac{dt}{t^{1+\frac{\sigma}{2}}} \bigg] \\[0.1cm]
        &= \frac{1}{|\Gamma(-\frac{\sigma}{2})|} h^2\frac{N}{12}\bigg[\|D^4 \phi\|_\infty\frac{1}{1-\frac{\sigma}{2}}  + \big( \|D^4 \phi\|_\infty  + \|\phi\|_\infty c_N \big)\frac{2}{\sigma}- \| \phi\|_\infty c_N \frac{1}{1+\frac{\sigma}{2}} \bigg] \\[0.1cm]
        & = h^2 \frac{N}{\cm 12\nc|\Gamma(-\frac{\sigma}{2})|}\frac{1}{\cm \frac \sigma2\nc}\Big(\frac{\|D^4 \phi \|_\infty}{\cm 1-\frac\sigma2\nc} +\frac{c_N\|\phi\|_\infty}{\cm 1+\frac\sigma2\nc}\Big).
    \end{align*}
\cm
Two applications of the identity $x\Gamma(x)=\Gamma(x+1)$ implies that 
$-\frac\sigma2(1-\frac\sigma2)\Gamma(-\frac\sigma2)=\Gamma(2-\frac\sigma2)$, and we can conclude by noting that for $\sigma\in(0,2)$, $\Gamma(2-\frac\sigma2)\in(0.88,1)$ and hence $12\,\Gamma(2-\frac\sigma2)\geq 10$.\nc
\end{proof}

From Lemma \ref{lem::discr_fraclap_quadrature} and {\cb Lemma \ref{lem:trun_err_fraclap} }it follows that $-(- \Delta_h)^{\frac{\sigma}{2}}$ is a monotone (positive coefficients) quadrature approximation of $-(- \Delta)^{\frac{\sigma}{2}}$ satisfying a discrete Levy condition. 

\subsection{Numerical approximation of the nonlocal PDE}
Approximating time derivatives by forward Euler and fractional Laplacians by powers of discrete Laplacians \eqref{discrete_fraclap}, we get the following explicit scheme:
\begin{equation}\label{scheme::simple_parabolic_eq}
\begin{split}
    U_{\bi}^{n+1} &= U_{\bi}^n+\tau \Big[ F\big(-(-\Delta_h)^{\frac{\sigma}{2}}U_{\bi}^n\big)+f_{\bi}^n\Big]\qquad\text{for}\qquad \bi\in\Z^N, n\in \mathcal{T}_\tau^T,\\
    U_{\bi}^{0} &= u_0 \qquad\text{for}\qquad \bi\in\Z^N,
\end{split}
\end{equation}
where $U_{\bi}^n=u(x_{\bi},t_n)$ and $f_{\bi}^n=f(x_{\bi},t_n)$ for $x_{\bi} \in \mathcal{G}_h$ and $t_n \in \mathcal{T}_\tau^T$. 

This scheme is monotone/satisfies a comparison principle under a CFL condition:
\begin{align}
    \tau \leq \frac{1}{L_F C_\sigma} h^{\sigma},\qquad \text{$C_\sigma$ and $L_F$ are given by Lemma \ref{lem::discr_fraclap_quadrature} 
and \ref{F1}}.
    \label{CFL_condition}
\end{align}
 
\begin{thm}[Comparison]\label{comp_S}
    Assume \ref{F1}, \eqref{CFL_condition}, and $U,V$ sub-, supersolutions of \eqref{scheme::simple_parabolic_eq}.\footnote{Subsolution means \eqref{scheme::simple_parabolic_eq} is satisfied as an $\leq$-inequality.} If $U_\bj^0 \leq V_\bj^0$ for all $\bj\in \mathbb{Z}^N$, then $U_\bj^n \leq V_\bj^n$ holds for all $n\geq1$ and $\bj\in \mathbb{Z}^N$.
    \label{thm::scheme_comparison}
\end{thm}
\begin{proof}
        Consider the following calculation:
        \begin{align*}
            &U_\bi^1-V_\bi^1 \\ 
            &\; {\cb=\ }  U_\bi^0-V_\bi^0 +\tau \Big( F\big( -(-\Delta_h)^{\frac{\sigma}{2}}U_\bi^0\big) - F\big( -(-\Delta_h)^{\frac{\sigma}{2}}V_\bi^0\big)\Big) \\
            &\leq U_\bi^0-V_\bi^0 + \tau L_F\Big( -(-\Delta_h)^{\frac{\sigma}{2}}\big(U_\bi^0-V_\bi^0\big)  \Big)^+ \\
            &= U_\bi^0-V_\bi^0 + \tau L_F\Big( \sum_{\bj \in \zn \setminus \{0\}} \Big( 
            \big(U_{\bi+\bj}^0-V_{\bi+\bj}^0\big)-\big(U_\bi^0-V_\bi^0\big)
            \Big) \kappa_{\sigma, h,\bj}  \Big)^+ 
        \end{align*}
\cb Since $U_\bj^0 - V_\bj^0\leq0$, and $a+b^+\leq (a+b)^+$ when $a\leq0$, we 
get\normalcolor
        \begin{align*}
             &{\cb U_\bi^1-V_\bi^1} \\
             &\leq  \tau L_F\Big( \sum_{\bj \in \zn \setminus \{0\}}  
            (U_{\bi+\bj}^0-V_{\bi+\bj}^0)\kappa_{\sigma, h,\bj}+(U_\bi^0-V_\bi^0)
              \Big(\frac{1}{\tau L_F}-\sum_{\bj \in \zn \setminus \{0\}}\kappa_{\sigma, h,\bj}\Big)  \Big)^+ \\
              &\leq 0,
        \end{align*}
        where the last inequality follows from the CFL-condition \eqref{CFL_condition}, \cb $\kappa_{\sigma, h,\bj}\geq0$ (Lemma \ref{lem::discr_fraclap_quadrature}), and $U_\bj^0-  V_\bj^0\leq0$.
        Hence $U_\bi^1 \leq V_\bi^1$ for all $\bi\in \mathbb{Z}^N$. The result follows by induction.\nc
 \end{proof}
From monotoncity/comparison, uniqueness and $L^{\infty}$-stablilty follow: 
\begin{thm}[Existence, uniqueness, and $L^\infty$-stability]\label{eus}
    Assume \ref{F1}, \ref{F3}, and CFL condition \eqref{CFL_condition}. Then there exists a unique solution $U$ of  \eqref{scheme::simple_parabolic_eq} satisfying
    $$\|U^n\|_{L^{\infty}} \leq \|u_0\|_{L^{\infty}} + t_n (|F(0)|+\|f\|_{L^{\infty}}).$$
\end{thm}
\begin{proof}
Existence \cb  and uniqueness are \nc immediate since the scheme is explicit. 
To show the $L^\infty$-bound, note that $\pm W$ are super/sub solutions of \eqref{scheme::simple_parabolic_eq} when $$W(x_\bj, t_n) = \|u_0\|_{L^{\infty}}+t_n (|F(0)|+\|f\|_{L^{\infty}}),$$
    since e.g.
        $W_\bj^{n+1}-W_\bj^n- \tau F\big((-\Delta_h)^{\frac{\sigma}{2}}W_\bj^n\big)
        = \tau (|F(0)|+\|f\|_{L^{\infty}}) - \tau F(0)
          \geq \tau f(x_\bj,t_n).$
    Then by comparison  (Theorem \ref{comp_S}), $-W\leq U\leq W$ and the result follows.
\end{proof}

\subsection{Convergence of the scheme} We will use the method of half-relaxed limits \cite{BS91}, and to do that we need to extend the scheme to the whole space and write it in a particular form so that we can verify the assumptions of the method. Let $U_h : \mathbb{R}^N \times [0,T] \rightarrow \mathbb{R}$ denote the solution and
\begin{align}
    S(h,x,t, U_h(x,t), U_h) = 0\qquad\text{in}\qquad  \mathbb{R}^N\times [0,T],
\label{eq::scheme_abstract_form}
\end{align}
the scheme on the whole space,
where $S$ is defined as
\begin{align}\label{S_definition}
    &S(h,x,t, r, U_h)  \\
    &\ =\ \begin{cases}
        \frac{r-U_h(x,t-\tau)}{\tau}-F\big(-(-\Delta_h)^{\frac{\sigma}{2}}U_h(x,t-\tau) \big)-f(x,t-\tau), & t\in[\tau,T],
        \\[0.2cm]
        r-u_0(x),& t\in[0,\tau). 
    \end{cases}\nonumber
\end{align}
The scheme \eqref{eq::scheme_abstract_form} and solution $U_h$ coinside with the scheme \eqref{scheme::simple_parabolic_eq} and solution $U$ when restricted to the grid $\mathcal{G}_h\times \mathcal{T}_\tau^T$.
Note that since $\tau = o(1)$ as $h \rightarrow 0$ below, we have skipped the depedence on $\tau$ in $S$. 

\begin{lem}[Existence, uniqueness, and $L^\infty$-stability]\label{stab}
Assume \ref{F1}, \ref{F3}, and CFL condition \eqref{CFL_condition}. Then there exists a unique solution $U_h$ of \eqref{eq::scheme_abstract_form} satisfying
$$\|U_h(t)\|_{L^{\infty}} \leq \|u_0\|_{L^{\infty}} + t (|F(0)|+\|f\|_{L^{\infty}}).$$
\end{lem}
\begin{proof}
Since at any point $(x,t)$, $U_h$ is the solution of \eqref{scheme::simple_parabolic_eq} on the grid $(x,t)+\mathcal G_h\times \mathcal T_\tau^T$, the result follows from Theorem \ref{eus}.
\end{proof}
\begin{rem}
    By comparison, $U_h$ will inherit continuity in $x$ from the data $f$ and $u_0$. It is also possible to show approximate continuity in time, $\sup_{|x|\leq R}|U_h(t,x)-U_h(s,x)|\leq \tilde\omega_R(|t-s|+\tau^{\frac12})$ for some modulus $\tilde\omega_R$.
\end{rem}

To show convergence, in addition to $L^\infty$-stability, we to show that the scheme in the form \eqref{eq::scheme_abstract_form} is monotone and stable in the sense of Barles-Souganidis \cite{BS91}.

\begin{lem}[BS monotone]\label{mono}
    Assume \ref{F1}, \ref{F3}, and CFL condition \eqref{CFL_condition}.
    The function $S(h,x,t,r,U_h)$ defined in \eqref{S_definition} is nondecreasing in $r$ and nonincreasing in $U_h$.
\end{lem}
\begin{proof}
Increasing in $r$ is immediate from the definition. Nonincreasing in $U_h$ follows from a direct computation using the properties of $-(-\Delta_h)^{\frac{\sigma}2}$, including positivity and estimates on the weights from Lemma \ref{lem::discr_fraclap_quadrature} combined with the CLF condition \eqref{CFL_condition}. The details are essentially given in the proof of comparison Theorem \ref{comp_S}.
\end{proof}

\begin{lem}[BS consistent]\label{thm:BS-consist}
 Assume \ref{F1}, \ref{F3}, $\tau = o(1)$ as $h \rightarrow 0$, and $S$ is defined in \eqref{S_definition}. If $\phi \in C^2 \cap C_b(\overline Q_T)$ and $\eta_h \geq 0$ is such that $\frac{\eta_h}{\tau} \rightarrow 0$ as $h \rightarrow 0$, then \footnote{$\displaystyle (u_0)_*(x)= \liminf_{y\to x} u_0(y)$ and $\displaystyle (u_0)^*(x)= \limsup_{y\to x} u_0(y)$.} 
    \begin{align*}
        &\liminf_{(h,t,x,\xi) \rightarrow (0,t_0, x_0, 0)} S(h,x,t, \phi(x,t)+\xi-\eta_h, \phi+\xi)
        \geq \\
        &\begin{cases}
            \phi_t(x_0,t_0)-F\big(-(-\Delta)^{\frac{\sigma}{2}}\phi(x_0,t_0)\big)-f(x_0,t_0), &
            t_0>0, \\[0.2cm]
            \min\Big\{\phi(x_0,0)-(u_0)^*(x),  \phi_t(x_0,0)-F\big(-(-\Delta)^{\frac{\sigma}{2}}\phi(x_0,0)\big)-f(x_0,0)\Big\}, & 
            t_0=0.
        \end{cases}
    \end{align*}
    and 
       \begin{align*}
        &\limsup_{(h,t,x,\xi) \rightarrow (0,t_0, x_0, 0)} S(h,x,t, \phi(x,t)+\xi-\eta_h, \phi+\xi)\leq \\
        &\begin{cases}
            \phi_t(x_0,t_0)-F\big(-(-\Delta)^{\frac{\sigma}{2}}\phi(x_0,t_0)\big)-f(x_0,t_0), &
            t_0>0, \\[0.2cm]
            \max\Big\{\phi(x_0,0)-(u_0)_*(x),  \phi_t(x_0,0)-F\big(-(-\Delta)^{\frac{\sigma}{2}}\phi(x_0,0)\big)-f(x_0,0)\Big\}, & 
            t_0=0,
        \end{cases}
    \end{align*}
\end{lem}
\begin{proof}
    The proof is similar to the proof of \cite[Lemma 5.5]{DJ24}, and we only to the $\limsup$-case since the $\liminf$-case is similar. First note that by Lemma \ref{eus} (and an approximation argument since $\phi\not\in C_b^4$), $(-\Delta_h)^{\frac{\sigma}{2}}\phi \rightarrow (-\Delta)^{\frac{\sigma}{2}}\phi $ locally uniformly as $h \rightarrow 0^+$ (see also \cite[Theorem 1.7]{Ciaurri-Stinga-02}). 
    Assume $(h,t,x,\xi)\rightarrow(0,t_0,x_0,0)$ and consider
    \begin{align*}
        &I:=S(h,x,t, \phi(x,t)+\xi+\eta_h, \phi+\xi).
    \end{align*}
    
    We start with the case $t_0 \in(0,T)$. For small enough $h$, both $t_0, t_0-\tau \in (\tau, T)$, so by \eqref{S_definition} $I$ becomes
    \begin{align*}
 &\tfrac{\phi(x,t)+\xi+\eta_h-(\phi(x,t-\tau)+\xi)}{\tau}-F\big(-(-\Delta_h)^{\frac{\sigma}{2}}(\phi(x,t-\tau)+\xi) \big)-f(x,t-\tau)\\[3pt]
        &\quad=\tfrac{\phi(x,t)-\phi(x,t-\tau)}{\tau}+\tfrac{\eta_h}{\tau}-F\big(-(-\Delta_h)^{\frac{\sigma}{2}}\phi(x,t-\tau) \big)-f(x,t-\tau).
    \end{align*}
Under our assumptions, we then find that
    \begin{align*}
        I\xrightarrow{h\rightarrow 0} \del_t\phi(x,t)-F\big(-(-\Delta)^{\frac{\sigma}{2}}\phi(x,t)\big)-f(x,t),
    \end{align*}
    and then that 
    \begin{align*}
     \lim_{h\to0} I\xrightarrow{(t,x,\xi)\rightarrow(t_0,x_0,0)} \del_t\phi(x_0,t_0)-F\big(-(-\Delta)^{\frac{\sigma}{2}}\phi(x_0,t_0)\big)-f(x_0,t_0).
    \end{align*}
    The result is in fact independent of the order of the limits.
    
    Now let $t_0=0$. Then $t$ can approach $t_0$ by points from either $[\tau, T)$ or $[0,\tau)$. Assume first $t \rightarrow t_0$ from $[\tau, T)$. As above we then find that
    \begin{align*}
\lim_{(h,t,x,\xi)\rightarrow(0,t_0,x_0,0)} I= \del_t\phi(x_0,0)-F\big(-(-\Delta)^{\frac{\sigma}{2}}\phi(x_0,0)\big)-f(x_0,0).
    \end{align*}
    If $t \rightarrow t_0$ from $[0,\tau)$, we instead find that 
    \begin{align*}  \limsup_{(h,t,x,\xi)\rightarrow(0,t_0,x_0,0)} I&= \limsup_{(h,t,x,\xi)\rightarrow(0,t_0,x_0,0)} \phi(x,t)+\xi+\eta_h-u_0(x) \\[3pt]
        &= \phi(x_0,0)-(u_0)_*(x_0).
    \end{align*}
Hence $\limsup_{(h,t,x,\xi)\rightarrow(0,t_0,x_0,0)} I$ is bounded above by the maximum of the two cases and the result follows.
\end{proof}

The main result of the paper shows that the scheme is convergent.
\begin{thm}[Convergence]
Assume \ref{F1}, \ref{F3}, CFL condition \eqref{CFL_condition}, and $u$ and $U_h$ solve \eqref{eq::simple_parabolic_eq} and \eqref{eq::scheme_abstract_form} respectively. 
    Then $U_h\to u$ locally uniformly as $h \rightarrow 0$.
\end{thm}
\begin{proof}
In view of our previous results, 
the convergence follows in a standard way from the Barles-Perthame-Souganidis method of half-relaxed limits \cite{BS91}. We sketch the proof, starting by introducing the ``half-relaxed limits" of $U_h$:
    \begin{align*}
        \overline{u}(x,t) = \limsup_{(y,s,h) \rightarrow (x,t,0^+)}U_h(y,s) \hspace{5mm}\text{and}\hspace{5mm} \underline{u}(x,t) = \liminf_{(y,s,h) \rightarrow (x,t,0^+)}U_h(y,s).
    \end{align*}
    By BS stability (Lemma \ref{stab})  $\overline{u}$ and $\underline{u}$ are finite and bounded, and then by stability of viscosity solutions and BS monotonicity and consistency (Lemmas \ref{mono} and \ref{thm:BS-consist}), $\overline{u}$ and $\underline{u}$ are sub- and supersolutions of \eqref{eq::simple_parabolic_eq} respectively. See the proof of \cite[Theorem 5.6]{DJ24} for a detailed proof in a case that is close to ours. By the strong comparison principle for equation \eqref{eq::simple_parabolic_eq} ({\cb Proposition} \ref{prop::comparison_existence_stability} $(i)$) we have that $\overline{u}\leq \underline{u}$. By definition, $\overline{u}\geq \underline{u}$, and hence the limsup and liminf are equal and the limit exists:
    \begin{align*}
        \lim_{(y,s,h) \rightarrow (x,t,0^+)}U_h(y,s) = \overline{u}(x,t) = \underline{u}(x,t) := u(x,t),
    \end{align*}
    and the limit $u$ is continuous and a viscosity solution of \eqref{eq::simple_parabolic_eq} (unique by {\cb Proposition} \ref{prop::comparison_existence_stability} $(ii)$).  Furthermore, $\limsup_{h \rightarrow 0^+}U_h \leq \overline{u} = \underline{u} \leq \liminf_{h \rightarrow 0^+}U_h$, so $U_h \rightarrow u$ pointwise as $h \rightarrow 0$. Local uniform convergence follows by the definitions of $\overline{u}$ and $\underline{u}$, see e.g. \cite[Chapter 5, Theorem 1.9]{Bardi-BOOK}.
\end{proof}

\cb
\begin{rem}
In viscosity solution theory,  initial and boundary conditions can be interpreted in a pointwise classical sense or in a generalized sense allowing for loss of boundary conditions and boundary discontinuities. The first is used in \cite{CJ17} and Proposition \ref{prop::comparison_existence_stability}, while the second is needed for the method of half-relaxed limits \cite{BS91}. However, for parabolic problems the two notions of initial conditions coincide when the initial data is continuous. Then $\R\times\{0\}$ is regular in the sense that the initial condition is continuously attained in every point. See e.g. \cite[Theorem 4.7]{Ba97} (or \cite[Lemma 4.2]{DJ24} for a nonlocal case), the proofs are essentially the same.
\end{rem}

Under our assumptions solutions of \eqref{eq::simple_parabolic_eq} need only be continuous (the problem may degenerate). Under stronger assumptions, the solution may be more regular and then it is possible to estimate also the convergence rate of the scheme. We refer to \cite{CJ23} for many results in this direction and a discussion  covering several special cases of problems \eqref{eq::simple_parabolic_eq} and \eqref{eqn:main}. Even if convergence rates are not a main focus of this paper, we will give one result showing that in the best possible case, our scheme is indeed 2nd order in space and first order time.

{\cb 
\begin{thm}[Convergence rate]
\label{thm::scheme_convergence_rate}
Assume \ref{F1}, \ref{F3}, $u \in C_b([0,T];C_b^{4}(\R^N))\cap C^2_b(0,T;C_b(\R^N))$, CFL condition \eqref{CFL_condition}, and $u$ and $U$ solve \eqref{eq::simple_parabolic_eq} and \eqref{scheme::simple_parabolic_eq}. Then
\begin{equation}
\begin{aligned}
    &\sup_{\bi\in \Z^N}|U_{\bi}^n - u(x_{\bi},t_n)| \\
    &\leq t_n L_F\Big(c_{N,\sigma,1}\|D^4 u(\cdot, t_n) \|_{L^\infty}+ c_{N,\sigma,2}\|u(\cdot, t_n) \|_{L^\infty}\Big)h^2 + t_n\| u_{tt}  \|_{L^\infty}\tau. 
\end{aligned}
\label{eq::scheme_convergence_rate}
\end{equation}
\end{thm}
\begin{rem}
\label{rem::convergence_rate_remark}

(a) The second order rate in $h$ is confirmed by our numerical experiments, see Section \ref{sec::example_4}. In our computations $\tau \leq O(h^2)$ is required to reach this rate as indicated by \eqref{eq::scheme_convergence_rate}. The scheme  still works under the milder CFL-condition   $\tau \leq O(h^\sigma)$ (cf. \eqref{CFL_condition}), but then the convergence is slower. 
  \smallskip

\noindent (b) \ Can $u$ ever be so smooth as in Theorem \ref{thm::scheme_convergence_rate}? The answer is yes if e.g. $F, f, u_0$ are smooth enough and $F'\geq c>0$ (strict ellipticity). A more precise discussion on this can be found in e.g. \cite{CJ23}.
\end{rem}
\begin{proof}
Let $u^n:=u(x,t_n)$, 
$u_{\bi}^n:=u(x_{\bi},t_n)$, 
etc. Evaluating \eqref{eq::simple_parabolic_eq} at time $t_n$ and adding and subtracting terms, we find that 
\begin{equation}
\begin{aligned}
    &\frac{u^{n+1}-u^{n}}{\tau} -  F(-(-\Delta_h)^{\sigma/2}u^n) \\
    &\qquad = f^n  +\underbrace{F(-(-\Delta)^{\sigma/2}u^n) - F(-(-\Delta_h)^{\sigma/2}u^n)}_{=:a}+\underbrace{\frac{u^{n+1}-u^{n}}{\tau}-u_t^n}_{=:b}.
\end{aligned}
\label{eq::convergence_order_calculation}
\end{equation}
Note that by \ref{F1} and Theorem \ref{lem:trun_err_fraclap}, and by Taylor expansion in time,
\begin{align*}
    |a | &\leq  L_F\big(c_{N,\sigma,1}\|D^4 u(\cdot, t_n) \|_{L^\infty}+ c_{N,\sigma,2}\|u(\cdot, t_n) \|_{L^\infty} \big)h^2 =: A \nonumber
    ,\\[0.2cm]
    |b|& \leq \| u_{tt} \|_{L^\infty} \tau =: B.
\end{align*}
Define now 
    $w^{\pm} := u \pm t(A + B )$. By \eqref{eq::convergence_order_calculation}, $w^{\pm}$ is a super/sub-solution of \eqref{scheme::simple_parabolic_eq}:
\begin{align*}
    &\frac{w^{\pm}(x_{\bi},t_{n+1})-w^{\pm}(x_{\bi},t_{n})}{\tau} - F(-(-\Delta_h)^{\sigma/2}w^{\pm}(x_{\bi},t_{n})) \\
    &= \frac{u_{\bi}^{n+1}-u_{\bi}^{n} }{\tau} - F(-(-\Delta_h)^{\sigma/2}u_{\bi}^{n}) +\frac{t_{n+1}-t_n}{\tau}(\pm (A+B)) \\ 
    &= f_{\bi}^n+ (a \pm A) + (b \pm B)\begin{array}{c}\geq \\ \leq\end{array}f_{\bi}^n.
\end{align*}
By the discrete comparison principle Theorem \ref{comp_S}, we then have that 
\begin{align*}
    w^-(x_{\bi},t_n) \leq U^n_{\bi} \leq w^+(x_{\bi},t_n).
\end{align*}
Rearranging and taking the max over $i\in \Z^N$ gives us \eqref{eq::scheme_convergence_rate}.
\end{proof}
}

{\cb 
\section{Numerical experiments}
\label{sec::num_exp}
In the following numerical experiments we will solve the scheme \eqref{scheme::simple_parabolic_eq} with $f = 0$, 
\begin{align}
\begin{aligned}
\label{eq::num_exp_scheme}
    U_i^{n+1}&=U_i^n+\tau F_m(-(-\Delta_h)^{\frac{\sigma}{2}}U_i^n), && m=1,2,3,\\
    U^0_i &= g_k(x_i),&& k=1,2,3,
\end{aligned}
\end{align}
where $(F_1, F_2, F_3)$ are 
\begin{align*}
     F_1(l) = \max(0,l), \qquad F_2(l)  = \max\bigg(\frac{1}{2}l,l\bigg), \qquad F_3(l) = l.
\end{align*}
Here $F_1, F_2$ are non-smooth, nonlinear and Lipschitz, $F_3$ is linear, smooth and Lipschitz, and $F_1$ is strongly degenerate, the others non-degenerate.
We consider 
$C^{1,1}$, non-smooth (Lipschitz), and $C^{\infty}$ initial conditions:
\begin{align*}
g_1(x)&= \begin{cases}
\frac34\sin(\pi(x+\frac32)) - \frac12\sin(\frac{\pi}2(x+1)) +\frac{1}4, & x\in(-2,2),\\[0.2cm]
0, & x\not\in(-2,2),
\end{cases}\\
g_2(x)&= \begin{cases}
2-|x|, & x \in (-2,-1]\cup [1,2), \\
2|x|-1, & x\in (-1,1), \\
0, & x\not\in(-2,2),
\end{cases}\\
g_3(x) &= \frac{1}{1+x^2}.
\end{align*}
\vspace{-0.3cm}
\begin{figure}[h]
    \centering
\includegraphics[width=1.0\textwidth]{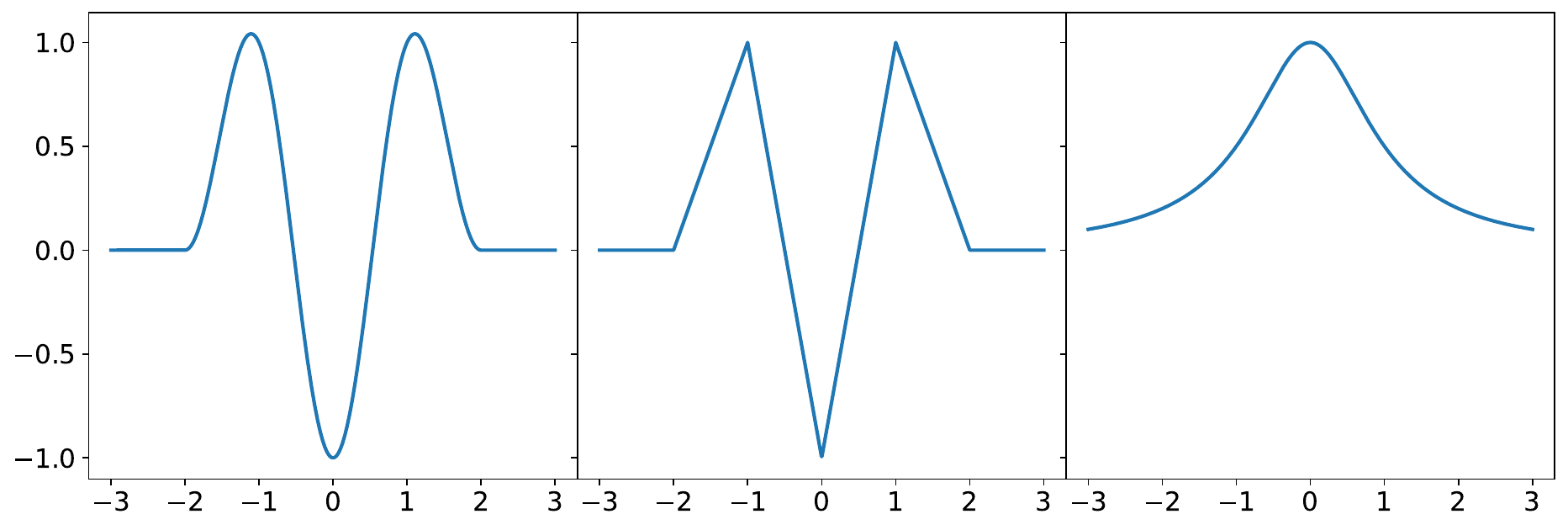}
    \caption{{\cb  Initial conditions $g_1$ (left), $g_2$ (middle), and $g_3$ (right).}}
    \label{fig::IC_fig}
\end{figure}

}

{\cb 
In all experiments we consider 1d fractional Laplacians, since in dimension $N=1$ the 
formula for the weights simplifies to (cf. \cite[Theorem 1.1 a)]{Ciaurri-Stinga-02}):
\begin{align}
    \kappa_{\sigma, h,j}= \frac{1}{h^\sigma}\frac{2^\sigma \Gamma(\frac{1+\sigma}{2})\Gamma(|j|-\tfrac{\sigma}{2})}{\sqrt{\pi}|\Gamma(-\frac{\sigma}{2})|\Gamma(|j|+1+\tfrac{\sigma}{2})}.
    \label{eq::weights_1d}
\end{align}
We can then compute the  fractional Laplacian as (cf. Lemma \ref{lem::discr_fraclap_quadrature})
\begin{align}
    -(-\Delta_h)^{\frac{\sigma}{2}}U_i^n = \sum_{j \in \mathbb{Z}\backslash \{i\}}\kappa_{\sigma, h,i-j}U_j^n - U_i^n\sum_{j \in \mathbb{Z}\backslash \{0\}}\kappa_{\sigma, h, j},
    \label{eq::num_exp_frac_lap}
\end{align}
truncating the sums at large $j$ to get finite sums.
However, without proper handling of the ratio of gamma-functions, this approach leads to poor numerical convergence.
A clever solution using asymptotic approximations and the $\zeta$-function is described by U. Unneberg in \cite[Appendix A]{ulrik_master_thesis}:
\begin{align}
    (-\Delta_h)^{\frac{\sigma}{2}}U_j \approx \mathcal L_{R, \sigma}[U]_j+\mathcal L^R_\sigma[U]_j, \qquad j\in \Z,
    \label{eq::ulrik_implementation}
\end{align}
where 
\begin{align*}
    \mathcal L_{R, \sigma}[U]_j &:= \frac{2^{\sigma}\Gamma(\frac{1+\sigma}{2})}{h^{\sigma}\sqrt{\pi}|\Gamma(-\frac{\sigma}{2})|}\sum_{\substack{|m|\leq R, \\ m\neq 0}} (U_j - U_{j-m}) \frac{\Gamma(|m|-\frac{\sigma}{2})}{\Gamma(|m|+1+\frac{\sigma}{2})}, \\    
    \mathcal L^R_\sigma[U]_j &:= \frac{2^{\sigma}\Gamma(\frac{1+\sigma}{2})}{h^{\sigma}\sqrt{\pi}|\Gamma(-\frac{\sigma}{2})|}2U_j\bigg(\zeta(1+\sigma)-\sum_{m=1}^{R-1}\frac{1}{m^{1+\sigma}} \bigg).
\end{align*}
We also use the (standard) trick of writing $\frac{\Gamma(|m|-\frac{\sigma}{2})}{\Gamma(|m|+1+\frac{\sigma}{2})}=e^{\ln(\Gamma(|m|-\frac{\sigma}{2}))-\ln(\Gamma(|m|+1+\frac{\sigma}{2}))}$, because we can then use larger $R$ before encountering $\infty/\infty$-type errors. When $\sigma=1$ we use that the ratio  simplifies:  $\frac{\Gamma(|m|-\frac{1}{2})}{\Gamma(|m|+1+\frac{1}{2})}=\frac{1}{(|m|+\frac{1}{2})(|m|-\frac{1}{2})}$.

To get a finite computational domain $I$, we truncate the domain,
introducing artificial $0$-Dirichlet exterior conditions on $I^c$.
}
\cb The relative $L^\infty$-error of a numerical solution $U$ with respect to 
a reference solution $\Tilde{U}$ is 
\begin{align*}
    \text{Rel. error} = \frac{\|U-\Tilde{U} \|_{L^\infty(I_{\text{error}})}}{\| \Tilde{U}\|_{L^\infty(I_{\text{error}})}},    
\end{align*} 
where 
$I_{\text{error}} \subset I$ is a smaller domain chosen to avoid pollution from the artificial 
exterior conditions.




 \nc
\subsection{Example 1: A degenerate diffusion equation in 1d}\label{sec:ex1} \ 
 Here we solve \eqref{eq::num_exp_scheme} with the degenerate nonlinearity $F_1$, {\cb  with $h=2^{-5}$, $\tau=O(h^{\sigma})$, and $I=[-20,20]$.}
\medskip

\noindent {\bf (1a)} \ We first look at a fixed time ${\cb t}=0.5$, two different initial conditions $g_1,g_2$, and three different values of $\sigma$ ($\frac12, 1, \frac32$), see Figure \ref{fig::exp1a_fig}. Compared to initial conditions (Figure \ref{fig::IC_fig}), we observe that the positive peaks remain fixed while the negative peak has diffused/traveled upwards. The solution never becomes smooth {\cb in the $g_2$-case.} 
\begin{rem}
Because of the the nonlinearity $F_1$, the solution can only move/diffuse when the fractional Laplacian is positive and this only happens if the solution is ``convex enough". At the positive peaks this is never the case, and the solution is stuck there. 
\end{rem}
\begin{figure}[h]
    \centering
    \includegraphics[width=1\textwidth]{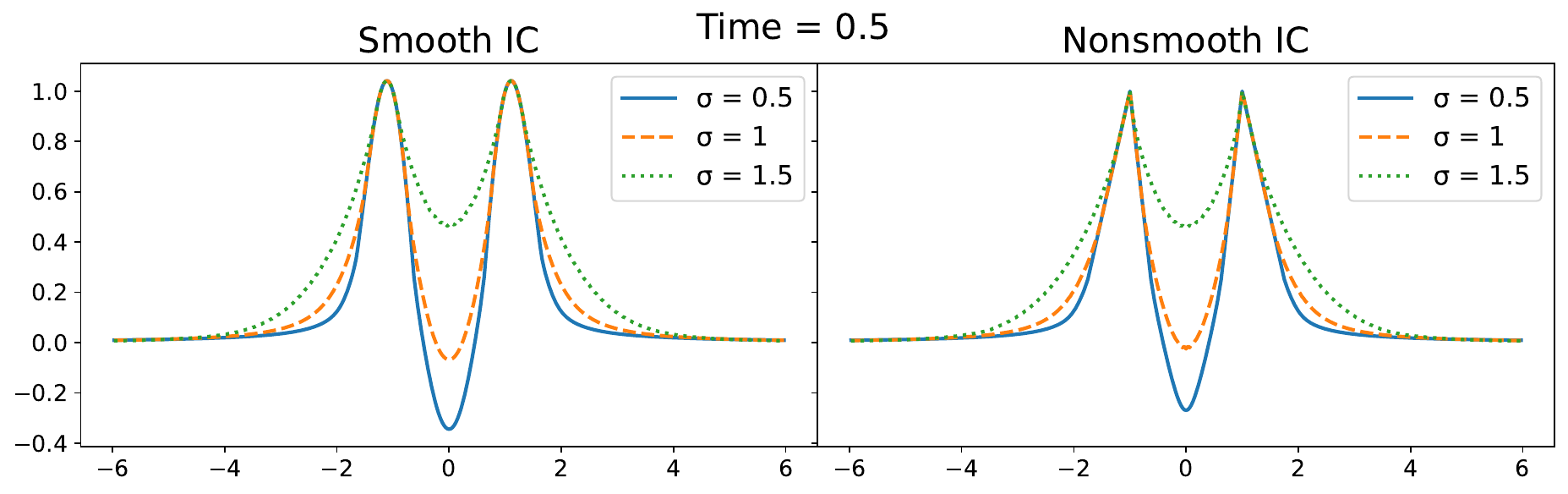}
    \caption{Experiment 1a. 
    %
    }
    \label{fig::exp1a_fig}
\end{figure}




\noindent {\bf (1b)} \ Now we study the time evolution of problem \eqref{eq::num_exp_scheme} with nonlinearity $F_1$ and non-smooth initial data $g_2$. We solve the problem at four different times and for three different values of $\sigma$ ($1,\frac32,2$), see Figure \ref{fig::exp1b_fig}. As expected, we observe the strongest diffusion in the most convex regions. 
\begin{rem}
    Short range diffusion becomes stronger the larger $\sigma$ is, but for long range diffusion it is the opposite. Therefore large $\sigma$ diffuse faster near peaks, but slower far away. This means that solutions in Figure \ref{fig::exp1b_fig} are not ordered, but rather they will reverse order far enough from the peaks.
\end{rem}
\begin{figure}[h]
    \centering
   \includegraphics[width=1\textwidth]{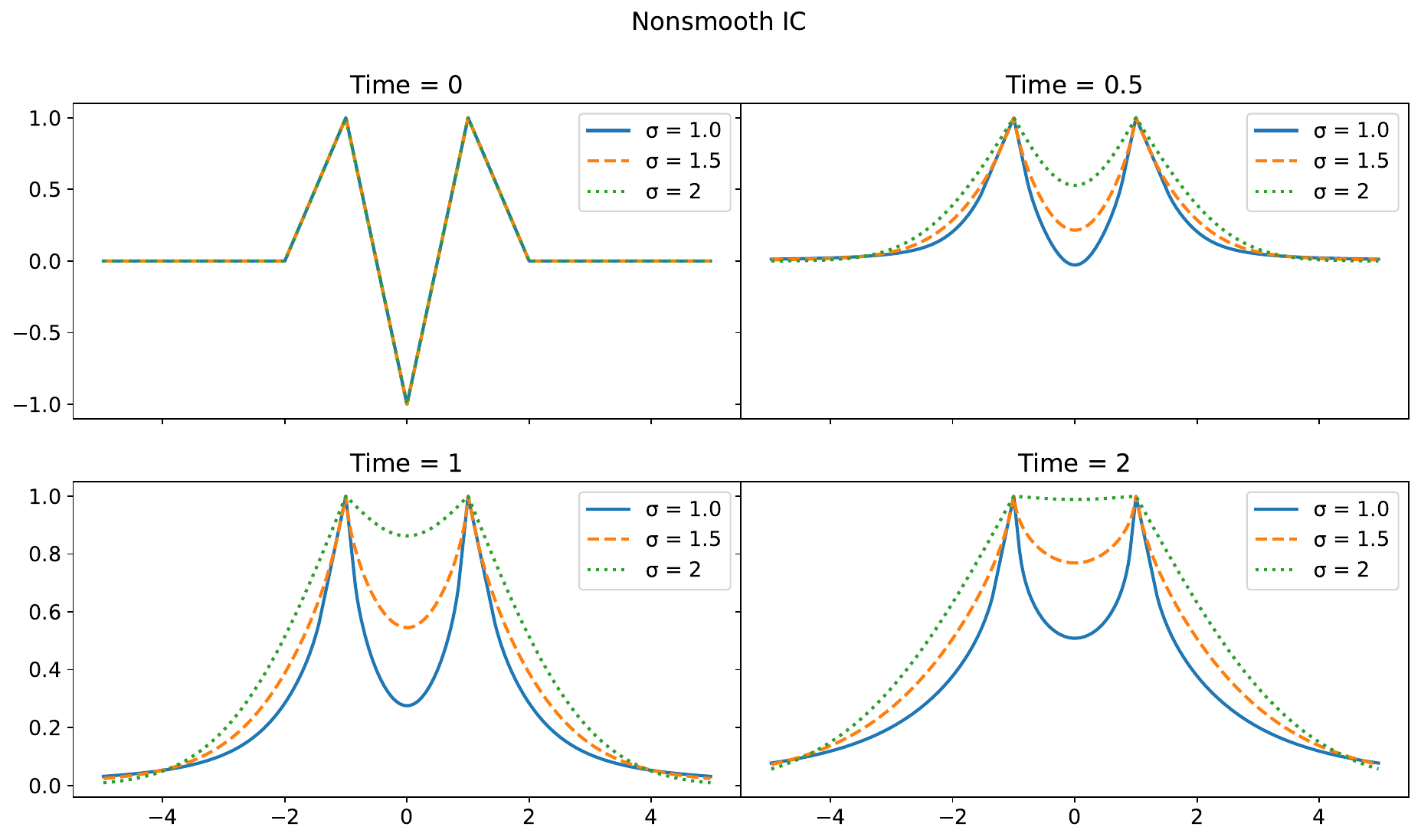}
    \caption{Experiment 1b. 
    }
    \label{fig::exp1b_fig}
\end{figure}








\subsection{Example 2: A degenerate diffusion equation in 2d}
We solve 
\begin{align}
\begin{split}
    U_{i,j}^{n+1}&=U_{i,j}^n+\tau \big[ F_1\big(-(-\Delta_h^x)^{\frac{1}{2}}U_{i,j}^n\big) +F_2\big(-(-\Delta_h^y)^{\frac{1}{2}}U_{i,j}^n\big) \big],\\
    U^0_{i,j} &= g_1\Big(\sqrt{x_i^2 +y_j^2}\Big),
\end{split}
    \label{eq::2d_num_exp_scheme}
\end{align}
where $(-\Delta_h^x)^{\frac{1}{2}}$ and $(-\Delta_h^y)^{\frac{1}{2}}$ denote (1d) $\frac12$-Laplacians in the $x$ and $y$ directions respectively {\cb (see Section \ref{sec:multdiff} below for the generalization to this class of equations)}. The initial condition is the 2d radially symmetric version of $g_1$. {\cb We use $h=2^{-5}$, $\tau=O(h^{{\sigma}})$, and $I=[-20,20]\times[-20, 20]$.} The solution at $t=1$ is shown Figure \ref{fig:fig_3d_layout}.
\begin{figure}[h]
    \centering   \includegraphics[width=\textwidth]{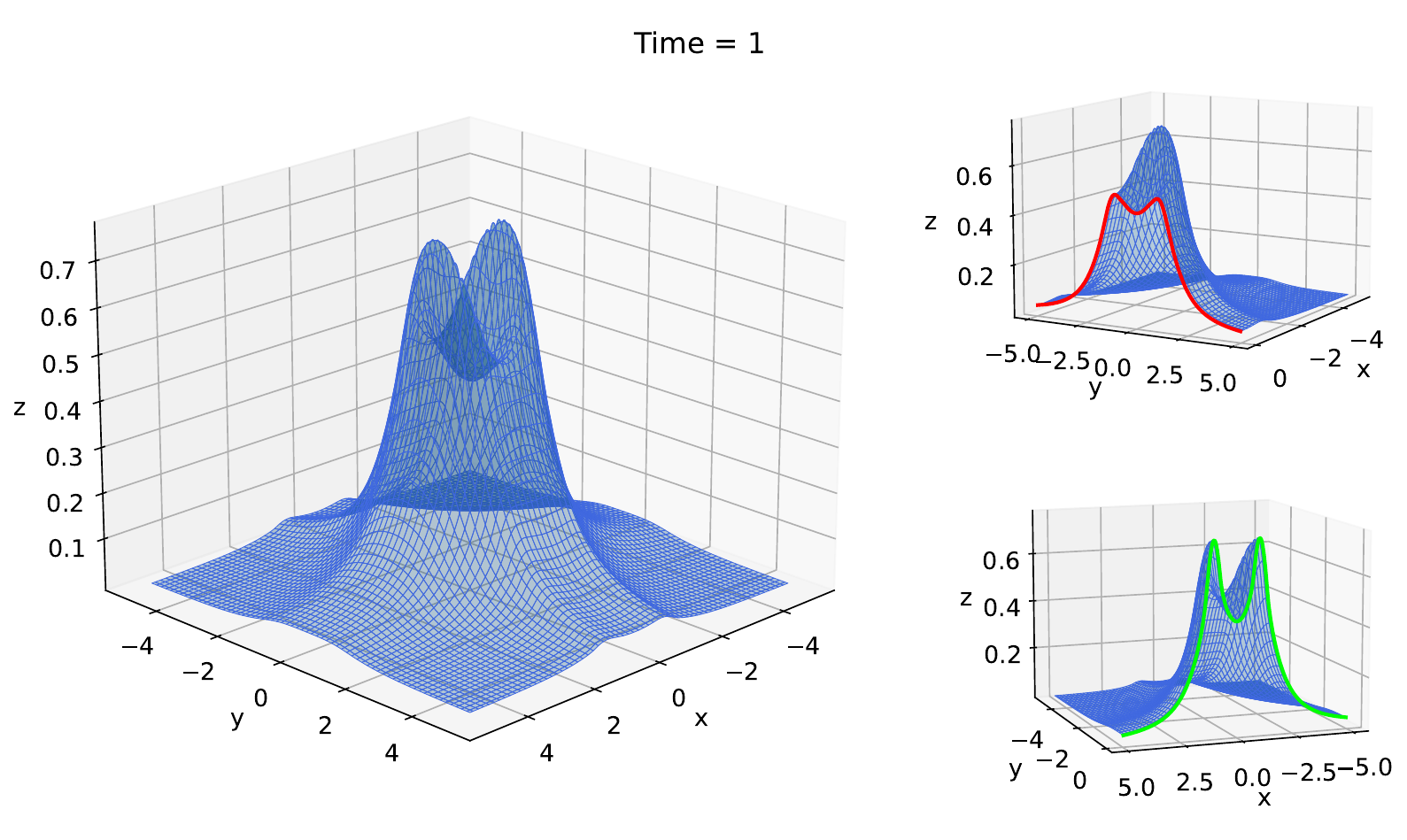}
    \caption{Experiment 2.  
    The figures to the right 
    show the solution for  $x \leq 0$ (top) and $y \leq 0$ (bottom). 
    }
    \label{fig:fig_3d_layout}
\end{figure}

In this example no points are stuck and the solution diffuses in the whole domain. Observe from the red and green curves how the two different nonlinearities in \eqref{eq::2d_num_exp_scheme} cause the solution to diffuse differently in the $x$ and $y$ directions. 

\subsection{Example 3: Limits of solutions as $\sigma\to0$ and $\sigma\to 2$}\
{ \cb  It is well-known that as  $\sigma\to 2^-$ and $\sigma\to 0^+$,  $-(-\Delta)^{\frac{\sigma}2}\to \Delta$  and $-(-\Delta)^{\frac{\sigma}2}\to -Id$,  and the correpsonding solutions of equation \eqref{eq::simple_parabolic_eq} converge to the solutions of their limiting equations \cite{CJ17}.  We study this convergence numerically when 
$F=F_1$ and $g=g_2$, with $h=2^{-5}$, $\tau=O(h^{\sigma})$, $I=[-20,20]$, and $I_{\text{error}} = [-10,10]$. 
The  solutions at $(t=1,\sigma\to0)$ and $(t=1,\sigma\to2)$ 
are shown in Figure \ref{fig:sigma_limit}.
We observe (see Table \ref{tab:sigma_to_2_only}) a linear convergence rate in $L^\infty$ with fixed $h$  for both $\sigma\to 0$ and $\sigma\to2$.
}

\begin{figure}[h]
    \centering
 \includegraphics[width=\textwidth]{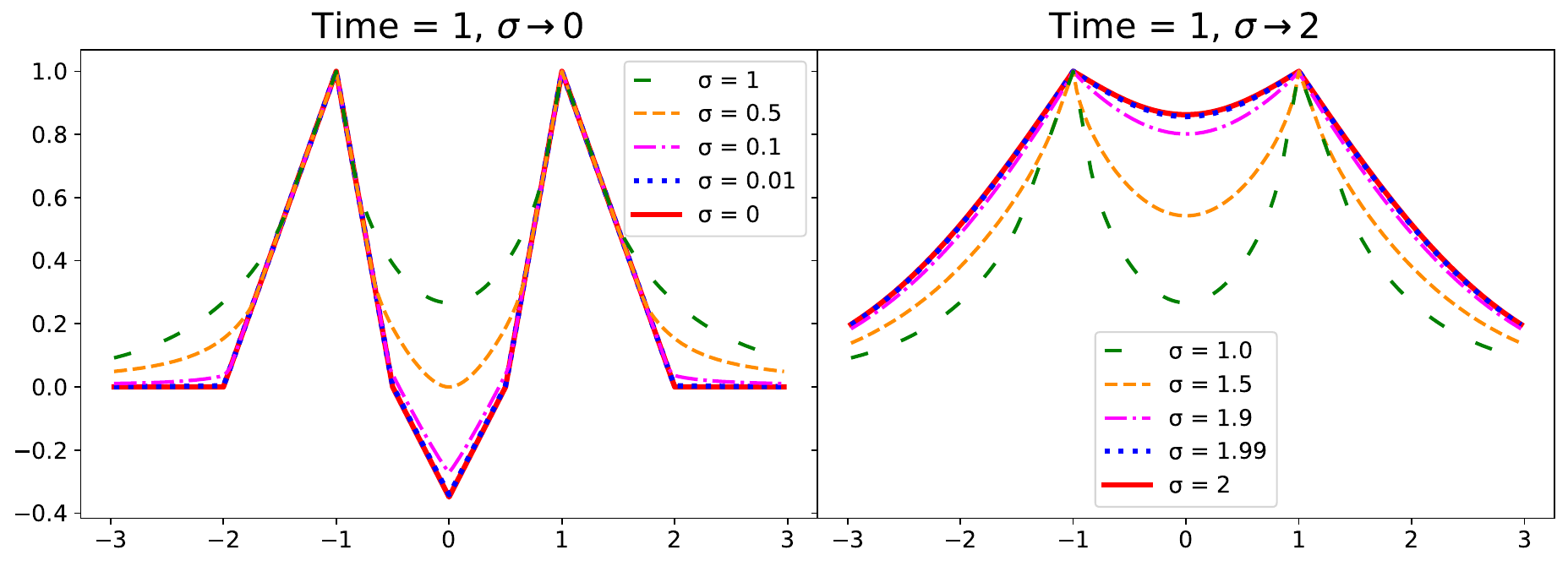}
    \caption{{\cb  Experiment 3. }}
    \label{fig:sigma_limit}
\end{figure}


\begin{table}[h]
    \centering
    \begin{tabular}{|ccc|}
\hline
       $\sigma$ &   Rel. error &       rate \\
\hline
 1.00e-1 &   7.83e-2  & --        \\
 5.00e-2 &   3.98e-2   &   0.98 \\
 2.50e-2 &   1.99e-2  &   1.00  \\
 1.25e-2 &   9.04e-3 &   1.14  \\
 6.25e-3 &   3.61e-3 &   1.32  \\
\hline
\end{tabular}
\hspace{-0.25cm}
\begin{tabular}{|ccc|}
\hline
     $2-\sigma$ &   Rel. error &      rate \\
\hline
 1.00e-1 &   6.8e-2  & --       \\
 5.00e-2 &   2.96e-2  &   1.04 \\
 2.50e-2 &   1.46e-2  &   1.02 \\
 1.25e-2 &   7.24e-3 &   1.01  \\
 6.25e-3 &   3.61e-3 &   1.01 \\
\hline
\end{tabular}
\bigskip
    \caption{ {\cb  Experiment 3, $\sigma\to0$ (left) and $\sigma\to2$ (right). { \cb Relative $L^{\infty}$-errors at time $t=1$ with fixed spatial step size $h=2^{-5}$.}}}
    \label{tab:sigma_to_2_only}
\end{table}




{\cb 
\subsection{Example 4: Convergence rate in $h$}
\label{sec::example_4}
We study the convergence rate as $h \rightarrow 0$ of the scheme \eqref{eq::num_exp_scheme} for $\sigma=1$ and $g=g_3$ on $I=[-5000, 5000]$. In Table \ref{tab:scheme_convergence_table} we list, for different values of $h$, the relative $L^\infty$-errors on $I_{\text{error}}=[-500,500]$ at time $t=1$ .
\medskip

\noindent {\bf (4a)} \ Linear equation with known exact solution. \ We reproduce Example $5.1$ from \cite{EJT18b}, 
solving \eqref{eq::num_exp_scheme} with $F=F_3$. 
In Table \ref{tab:scheme_convergence_table}, the relative $L^\infty$-errors are with respect to the exact solution \cite{EJT18b}
\begin{align}
    \label{eq::linear_eq_exact_sol}
    u(x,t)=\frac{t+1}{(t+1)^2+x^2}
\end{align}
as reference solution. The results are consistent with the $O(h^2+\tau)$ bound of Theorem \ref{thm::scheme_convergence_rate} and the discussion in Remark \ref{rem::convergence_rate_remark}. 
\medskip

\noindent {\bf (4b)} \ Nonlinear equation, numerical reference solution. 
We solve \eqref{eq::num_exp_scheme} with nonlinearity $F=F_2$. In Table \ref{tab:scheme_convergence_table}, the relative $L^\infty$-errors are with respect to the numerical solution with $h=2^{-7}$ as reference solution.

\begin{table}[h]
    \centering
    \begin{tabular}{|c|cc}
    \hline
    & & \\[-0.3cm]
    & $F_3$ (linear) &  $\tau=h$ \\[0.1cm]
\hline \\[-0.4cm]
        h &   Rel. error &       rate \\
\hline \\[-0.3cm]
 $2^{-1}$      &   1.20e-1   & --        \\
 $2^{-2}$     &   6.37e-2  &   0.91 \\
 $2^{-3}$    &   3.17e-2  &   1.01  \\
 $2^{-4}$   &   1.57e-2  &   1.01  \\
 $2^{-5}$  &   7.84e-3   &   1.01  \\
 $2^{-6}$ &   3.91e-3 &   1.00   \\
\hline
\end{tabular}\hspace{-0.12cm}
\begin{tabular}{|cc}
\hline
 & \\[-0.3cm]
$F_3$ (linear) & $\tau=h^2$ \\[0.1cm]
\hline\\[-0.4cm]
    Rel. error &      rate \\
\hline \\[-0.3cm]
   5.91e-2   & --       \\
   1.39e-2   &   2.08 \\
   3.44e-3  &   2.02 \\
   8.56e-4 &   2.01 \\
  2.14e-4 &   2.00 \\
   5.34e-5  &   2.00 \\
\hline
\end{tabular}\hspace{-0.12cm}
    \begin{tabular}{|cc}
    \hline
    & \\[-0.3cm]
$F_2$ (nonlinear) & $\tau=h^2$ \\[0.1cm] 
\hline\\[-0.4cm]
   Rel. error &      rate \\
\hline \\[-0.3cm]
  2.02e-2   & --       \\
  4.77e-3  &   2.08 \\
  1.17e-3   &   2.03 \\
  2.88e-4 &   2.02 \\
  6.85e-5  &   2.07 \\
  1.37e-5 &   2.32 \\
  \hline 
\end{tabular}
\hspace{-0.12cm}\vrule \hspace{-0.23mm}\vrule
\bigskip
    \caption{ {{\cb  Experiment 4a (left, middle) and 4b (right), $\sigma=1$.
    \cb Relative $L^{\infty}$-errors at time $t=1$.}}}
    \label{tab:scheme_convergence_table}
\end{table}
}
\section{Extensions}
\label{sec::extn}

\subsection{On time discretizations}

Different time discretizations can be considered as long as the resulting schemes are monotone. Here we discuss the $\theta$-method and the corresponding scheme for 
\eqref{eq::simple_parabolic_eq}:
\begin{equation}\label{scheme::simple_parabolic_eq-theta}
    U_{\bi}^{n+1} = U_{\bi}^n+\tau \Big[ F\Big(-(1-\theta)(-\Delta_h)^{\frac{\sigma}{2}}U_{\bi}^n - \theta(-\Delta_h)^{\frac{\sigma}{2}}U_{\bi}^{n+1} \Big)+f_{\bi}^n\Big].  
\end{equation}
This scheme is fully implicit when $\theta=1$. It is explicit when $\theta=0$ and then the scheme coincides with \eqref{scheme::simple_parabolic_eq}. For other values, the scheme has both explicit and implicit terms.
The scheme is monotone/$L^\infty$-stable under a modified CFL condition
\begin{align}
    \quad (1-\theta)\tau \leq \frac{1}{L_F C_\sigma} h^{2\sigma},\quad \text{$C_\sigma$ and $L_F$ are given by Lemma \ref{lem::discr_fraclap_quadrature} and \ref{F1}}.
    \label{CFL_condition-theta}
\end{align}
 Comparison, $L^{\infty}$-stability, and consistency follow from similar arguments as for the explicit scheme \eqref{scheme::simple_parabolic_eq}, see Theorems \ref{comp_S}, \ref{eus} and \ref{thm:BS-consist}. Existence is no longer immediate, but follows 
 from a fixed point argument (Banach) for each time step followed by an induction on the time steps -- for a detailed proof see e.g. the arguments for \cite[Theorem 3.1]{BJK1}. Convergence then follows as before:
 \begin{thm}
 Assume \ref{F1}, \ref{F3}, $\theta \in [0,1]$, CFL condition \eqref{CFL_condition:Isaacs}, and $u$ and $U_h$ solve \eqref{eq::simple_parabolic_eq} and \eqref{scheme::simple_parabolic_eq-theta} respectively. Then $U_h \to u$ locally uniformly as $h \rightarrow 0$.
\end{thm}
\begin{rem}
Only the fully implicit schemes $\theta=1$ have no CFL restriction here, see \eqref{CFL_condition-theta}. Here the CFL condition is a condition to be monotone/$L^\infty$-stable.
For linear (local) problems the case $\theta=\frac12$ is known as the Cranck-Nicholson scheme. This scheme is von Neumann stable without any CFL condition, but to have stronger $L^\infty$-stability a CFL condition is also needed. 
\end{rem}

\subsection{Equations with first order term}
To explain how first order terms and convection phenomena modify our schemes and analysis, we consider the equation
\begin{equation}
\begin{aligned}
    u_t -F\big(-(-\Delta)^{\frac{\sigma}{2}}u\big)+H(Du)&=f(x,t),  
\end{aligned}    
\label{eq:conv}
\end{equation}
where $H:\rn \to \R$ satisfies 
\medskip
\begin{enumerate}
\myitem{$\mathbf{(H1)}$}\label{H1} $|H(p_1) - H(p_2)| \leq L_H |p_1 -p_2|$ for  $p_1, p_2 \in \rn$.
\end{enumerate}
\medskip
Under assumptions \ref{F1}, \ref{F3}, \ref{H1}, comparison and wellposedness of \eqref{eq:conv} holds as in Proposition \ref{result:welposed-visco}. The proof remains the same since the results of \cite{CJ17} are very general and cover this case as well.

 There are many ways to discretize  the $H(Du)$-term to get a monotone numerical scheme for \eqref{eq:conv}. Such schemes are often derived from monotone conservative schemes for scalar conservation laws like e.g. up-wind, Lax-Freidrich, Gudonov or Enquist-Osher type of schemes \cite{Crandall-lions, Osher-Shu,Osher-Sethian,Bardi-Osher}. For simplicity we consider an explicit scheme based on the Lax-Friedrich discretization \cite[Section 2]{Crandall-lions}: 
\begin{equation}\label{scheme::eq_conv}
    U_{\bi}^{n+1} = U_{\bi}^n+\tau \Big[ F\big(-(-\Delta_h)^{\frac{\sigma}{2}}U_{\bi}^n\big) - H (\grad^{\text{cd}}_h U_{\bi}^{n})+  h\, \Delta^H_{h} U_{\bi}^n+f_{\bi}^n\Big],
\end{equation}
where for $k\in \{1, \cdots, n\}$,
\begin{align*}
    &\grad^{\text{cd}}_h U = \big(D^c_{h,1}U,\cdots, D^c_{h,n}U\big) \quad \text{with} \quad D^c_{h,k}U(x) =\tfrac{U(x+ he_k) - U (x -he_k)}{2h},\\
    &\Delta^H_h U = \sum_{k=1}^N L^k_{H}  \tfrac{U(x+h e_k)  -2 U(x)+ U(x-he_k)}{h^2} \quad \text{with} \quad L^k_H = \tfrac{\|\partial_k H\|_{L^{\infty}}}{2}.
 \end{align*} 

This scheme is monotone under the CFL condition
\begin{align}
    \tau \leq \frac{1}{2h^{-1} L_H+ L_F 
    C_{\sigma}h^{-\sigma}} \qquad \text{where} \qquad  L_H = \sum_{k=1}^N L^k_H,
    \label{CFL_condition:eq-conv}
\end{align}
$L_F$ as in \ref{F1}, and $C_\sigma$ as in Lemma \ref{lem::discr_fraclap_quadrature}. A simpler sufficient condition is $\tau \leq C \min\{h, h^{\sigma}\}$ for a $C$ depending on $L_H$ and $L_F$.  

Under CFL-condition \eqref{CFL_condition:eq-conv} (and \ref{F1}, \ref{F3}, \ref{H1}), comparison for \eqref{scheme::eq_conv} holds as in Theorem \ref{thm::scheme_comparison}, and then well-posedness and $L^{\infty}$-stability of \eqref{scheme::eq_conv} follow as before without change of proofs. Since $ H (\grad^{\text{cd}}_h \phi)+  h\, \Delta^H_{h} \phi \to H (\grad \phi)$ as $h\to 0$ for smooth functions $\phi$, the scheme is consistent and a version of Theorem \ref{thm:BS-consist} follows. By the method of half-relaxed limit we then get the following convergence result:
\begin{thm}
    Assume \ref{F1}, \ref{F3}, \ref{H1}, CFL condition \eqref{CFL_condition:eq-conv}, and  $u$ and $U_h$ solve \eqref{eq:conv} and \eqref{scheme::eq_conv} respectively. Then $U_h \to u$ locally uniformly to as $h \rightarrow 0$.
\end{thm}

\cb 
\subsection{Equations with more general diffusion terms}\label{sec:multdiff}
More complicated diffusions can be considered, like sums of terms with fractional Laplacians of different order and dimension. Many extensions are possible, here we just discuss one:
\begin{equation}
\begin{aligned}
    u_t -F_1\big(-(-\Delta_{\alpha^1})^{\frac{\sigma_1}{2}}u\big)-\dots-F_P\big(-(-\Delta_{\alpha^P})^{\frac{\sigma_P}{2}}u\big)&=f(x,t),  
\end{aligned}    
\label{eq:multdiff}
\end{equation}
where $P\in\N$, $\sigma_k\in(0,2)$ and $\Delta_{\alpha^k}=\alpha_1^k\partial_{1}^2+\dots+\alpha^k_{N}\partial_{N}^2$, $\alpha^k_j\in\{0,1\}$, for $k=1,\dots,P$ and $j=1,\dots,N$. E.g., $\Delta_{(0,1,0)}^{1/2}=(\partial_2^2)^{\frac12}$ and $\Delta_{(1,0,1)}^{1/3}=(\partial_1^2+\partial_3^2)^{\frac13}$.
Under assumptions \ref{F1} (for each $F_k$) and \ref{F3},  comparison and wellposedness of \eqref{eq:multdiff} holds as in Proposition \ref{result:welposed-visco} (the results of \cite{CJ17} still applies).

We discretise as before by replacing fractional Laplacians by powers of discrete Laplacians and using forward Euler in time:
\begin{equation}\label{scheme::eq_multdiff}
    U_{\bi}^{n+1} = U_{\bi}^n+\tau \Big[ F_1\big(-(-\Delta_{\alpha^1,h})^{\frac{\sigma_1}{2}}U_{\bi}^n\big) + \dots+F_P\big(-(-\Delta_{\alpha^P,h})^{\frac{\sigma_P}{2}}U_{\bi}^n\big)+f_{\bi}^n\Big],
\end{equation}
This scheme is monotone under a CFL condition, e.g.
\begin{align}
\tau \leq  \min_{k=1,\dots,P}\frac{h^{\sigma_k}}{L_{F_k}C_{\sigma,k}}
    \label{CFL_condition:eq-multdiff}
\end{align}
where $L_{F_k}$ is defined in \ref{F1}, and $C_{\sigma,k}$ is as in Lemma \ref{lem::discr_fraclap_quadrature}. 
Under CFL-condition \eqref{CFL_condition:eq-multdiff} (and \ref{F1}, \ref{F3}), comparison for \eqref{scheme::eq_multdiff} holds as in Theorem \ref{thm::scheme_comparison}, and then well-posedness and $L^{\infty}$-stability of \eqref{scheme::eq_conv} follow as before without change of proofs. The scheme is consistent and a version of Theorem \ref{thm:BS-consist} follows. By the method of half-relaxed limit we then get the following convergence result:
\begin{thm}
    Assume \ref{F1}, \ref{F3},  CFL condition \eqref{CFL_condition:eq-multdiff}, and  $u$ and $U_h$ solve \eqref{eq:multdiff} and \eqref{scheme::eq_multdiff} respectively. Then $U_h \to u$ locally uniformly to as $h \rightarrow 0$.
\end{thm}

\nc

\subsection{More general equations }
Here we discuss how to extend our results to the nonlocal HJB/Isaacs equations \eqref{eqn:main}. In view of e.g. \cite{CJ17,JK05}\footnote{The remarks of footnote 1 page 4 still apply.} sufficient conditions for strong comparison and well-posedness are given by 
\medskip
\begin{enumerate}
\myitem{$\mathbf{(I1)}$}\label{I1} 
 $c^{\A,\B}, a^{\A,\B}, b^{\A,\B}, f^{\A,\B}$ are continuous in $\A, \B, t, x$, and $a^{\A,\B}, c^{\A,\B} \geq 0$, and 
 $$\sup_{\A,\B} \Big\{ \|(a^{\A,\B})^{\frac1\sigma}\|_{W^{1,\infty}}, \|b^{\A,\B}\|_{W^{1,\infty}}, 
 \|c^{\A,\B}\|_{L^{\infty}},
\|f^{\A,\B}\|_{L^{\infty}}\Big\} \leq K. $$
 \end{enumerate}
 The assumption on $a$ is best understood by looking at section \ref{sec:DG}. Assumption \ref{I1} implies that the coefficients of the underlying SDE \eqref{SDE} are Lipschitz.
 
Consider now the following explicit scheme for \eqref{eqn:main},  
\begin{align}
    \label{scheme:isaacs}
    U_{\bi}^{n+1} = U_{\bi}^n + \tau \inf_{\B\in\mathcal{B}}\sup_{\A\in\mathcal{A}} \big\{ f^{\alpha, \beta} (x_{\bi}, t_n) - c^{\alpha, \beta} (x_{\bi}, t_n)  U^n_{\bi} + \LL_{h}^{\alpha, \beta} U_{\bi}^n   \big\}, 
\end{align}
where
\begin{align*}
\LL_{h}^{\alpha, \beta} U_{\bi}^n =&  - a^{\alpha, \beta} (x_{\bi}, t_n) (- \Delta_h)^\frac{\sigma}{2} U_{\bi}^n  \\ 
 & \qquad + \sum_{k=1}^N \Big(b^{\alpha, \beta, +}_k (x_{\bi}, t_n) D_{h,k}^+ U_{\bi}^n + b^{\alpha, \beta, -}_k (x_{\bi}, t_n) D_{h,k}^- U_{\bi}^n \Big) , 
\end{align*}
and $D_{h,k}^{\pm} \phi(x) = \frac{\phi(x\pm h e_k) -\phi(x)}{h}$ and $b^{\alpha, \beta} = (b^{\alpha, \beta}_1, \cdots, b^{\alpha, \beta}_N)$. Here we have used an upwind approximation of the gradient term.

The scheme is monotone if for all gridpoints $(x_{\bi}, t_n)$,
\begin{align*}
    \tau \leq \frac{1}{h^{-1}\sum_{k=1}^N |b^{\alpha, \beta}_k(x_{\bi}, t_n)|+ a^{\alpha, \beta} (x_{\bi}, t_n) C_{\sigma} h^{-\sigma} + c^{\alpha, \beta} (x_{\bi}, t_n)}. 
\end{align*}
By \ref{I1}, a sufficient CFL condition is then given by
\begin{align}
\label{CFL_condition:Isaacs}
    \tau \leq \frac{1}{K \big(Nh^{-1} + C_{\sigma} h^{-\sigma} + 1\big)}, \qquad \text{where} \  K \ \text{is given by \ref{I1}}.
\end{align} 
 Under CFL condition  \eqref{CFL_condition:Isaacs} (and \ref{I1}, \ref{F3}),  it is straight forward to show comparison of the scheme (c.f. Theorem \ref{comp_S}). $L^{\infty}$-stability then follows by taking $\pm \big(\|u_0\|_{L^{\infty}} + t_n \sup_{\A, \B}\|f^{\A,\B}\|_{L^{\infty}}\big)$ as sub- and supersolutions. Again it is easy to verify that $\LL_{h}^{\alpha, \beta} \phi \to \LL^{\alpha, \beta} \phi$ as $h\to 0$ for any smooth bounded function $\phi$, and consistency a la Theorem \ref{thm:BS-consist} follow by similar arguments. By the half-relaxed limit method we then again have a convergence result.
\begin{thm}
    Assume \ref{F3}, \ref{I1}, CFL condition \eqref{CFL_condition:Isaacs}, and  $u$ and $U_h$ solve \eqref{eqn:main} and  \eqref{scheme:isaacs} respectively. Then $U_h\to u$ locally uniformly as $h \rightarrow 0$.
\end{thm}

\begin{rem}
    Our results also apply when  $\sigma$ depends on $\A,\B$. In this case the differential operator in equation \eqref{eqn:main} is 
    $$\LL^{\alpha, \beta} \phi = -a^{\A,\B}(x,t)(-\Delta)^{\frac12\sigma^{\A,\B}}\phi(x) + b^{\A,\B}(x,t)\cdot D\phi(x),$$
    for $\sigma^{\A,\B} \in (0,\bar\sigma]$ and $\bar\sigma<2$. 
    In this case the CFL condition depends on the maximal order $\bar \sigma$ of the fractional Laplacian operators, $\tau \leq \frac{1}{K(Nh^{-1} + C_{\bar \sigma}h^{-\bar \sigma} +1)}$.   
\end{rem}
\begin{rem} Another extension is to consider equations involving powers of more general 2nd order elliptic differential operators $L$. Such powers are defined form formula \eqref{fraclap} by replacing $\Delta$ with $L$, and the idea is to approximate by replacing $L$ by a finite difference approximation $L_h$. Such methods could be analysed in a similar way as we do here.
\end{rem}






\section*{Declarations}
\subsection*{Ethical approval}\addtocontents{toc}{\SkipTocEntry}
Not applicable.

\subsection*{Competing interests}\addtocontents{toc}{\SkipTocEntry}
The authors declare that they have no competing interests or other interests that might be perceived to influence the results and/or discussion reported in this paper.

\subsection*{Authors' contributions}\addtocontents{toc}{\SkipTocEntry}
R.Ø.L. did the numerical simulations. All authors contributed equally to the rest of the paper.

\subsection*{Funding}\addtocontents{toc}{\SkipTocEntry}
I.C. was supported by DST-India funded INSPIRE faculty fellowship (IFA22-MA187). E.R.J. and R.Ø.L. received funding from the Research Council of Norway under Grant Agreement No. 325114 “IMod. Partial differential equations, statistics and data: An interdisciplinary approach to data-based modelling”.

\subsection*{Availability of data and materials}\addtocontents{toc}{\SkipTocEntry}
All the data have been generated by numerical simulations coded in Python. This code can be shared by R.Ø.L. upon request.

\end{document}